\documentclass[12pt]{article}
\usepackage{hyperref}
\usepackage{graphicx}
\usepackage{amssymb,amsmath,amsthm}
\usepackage{bbold}
\usepackage{bm}
\def\RPt{{\bR {\mathrm P}^2}}
\def\QPt{{\bQ {\mathrm P}^2}}
\def\emb{\hookrightarrow}

\def\mucube{\hbox{$\mu$-cube}}
\def\mupar{\hbox{$\mu$-parallelepiped}}
\def\mupars{\hbox{$\mu$-parallelepipeds}}

\def\irr{\mathop{irr}}

\def\sp{\mathop{span}}

\def\rk{\mathop{rk}}
\def\zone{\cD}
\def\fc{{\frak c}}
\def\fd{{\frak d}}
\def\cE{{\cal E}}
\def\rational{{\bQ}}

\def\bH{{\boldsymbol{B}}}
\def\bT{{\mathbb T}}

\def\bZ{\mathbb{Z}}
\def\bC{\mathbb{C}}
\def\bQ{\mathbb{Q}}
\def\bR{\mathbb{R}}
\def\bS{\mathbb{S}}

\def\cH{{\cal H}}
\def\cM{M}

\def\cC{{\cal C}}
\def\cD{{\cal D}}
\def\cE{{\cal E}}

\def\cK{{\cal K}}
\def\cH{{\cal H}}

\def\cM{{\cal M}}

\def\cF{{\cal F}}
\def\cU{{\cal U}}

\def\real{\bR}

\newtheorem{definition}{Definition}
\newtheorem{theorem}{Theorem}

\newtheorem{example}{Example}

\newtheorem{remark}{Remark}

\newtheorem{proposition}{Proposition}
\newtheorem{conjecture}{Conjecture} 

\usepackage{array}
\newcolumntype{S}{>{\centering\arraybackslash} m{.475\linewidth}}
\newcolumntype{T}{>{\centering\arraybackslash} m{10.5cm}}
\newcolumntype{U}{>{\centering\arraybackslash} m{1.5cm}}
\title{%
  A survey on quasiperiodic topology
}
\author{Roberto De Leo}
\begin{document}
\maketitle
\begin{abstract}
  This article is a survey of the Novikov problem of the structure of leaves of the foliations induced by a collection of closed 1-forms in a compact manifold $M$.
  Equivalently, this is to the study of the level sets of {\em multivalued} functions on $M$. To date, this problem was thoroughly investigated only for
  $M=\bT^n$ and multivalued maps $F:\bT^n\to\bR^{n-1}$ in three different particular cases: when all components of $F$ but one are multivalued,
  started by Novikov in 1981; when all components of $F$ but one are singlevalued, started by Zorich in 1994; when none of the components
  is singlevalued, started by Arnold in 1991. The first two problems can be formulated as the study of the level sets of certain quasiperiodic functions,
  the last as level sets of pseudoperiodic functions. In this survey we present the main analytical and numerical results to date and some physical phenomena
  where they play a fundamental role.
\end{abstract}
\section{Introduction}
About a century ago, M. Morse~\cite{Mor25} discovered a deep, beautiful and far-reaching relation -- what we now call {\em Morse theory}~\cite{Mil63} --
between topology and analysis, namely the fact that many fundamental topological information are encoded in the set of critical points of each
``generic enough'' function (Morse function). In other words, given a topological space $M$, several of its fundamental properties -- that, thanks to a celebrated
result of Gelfand and Naimark~\cite{GN47}, are known to be fully encoded inside $C(M)$ -- can also be extracted from elements of the space
of much more regular functions $C^\infty(M)$.

We recall that Morse theory is a much more powerful tool for the category of {\em compact} manifolds than for the one of non-compact ones,
due to the much larger geometrical and topological freedom the latter ones enjoy. Nevertheless, it is reasonable to think that there
exist some ``well-behaving'' classes of non-compact manifolds that are close enough, in some sense, to the compact case that some strong 
general geometrical and topological property holds for them. A trivial example of such class is the set of periodic submanifolds of $\bR^n$,
since they are all regular covers of  compact submanifolds of $\bT^n$ and therefore can be treated with the (more powerful) tools of
Morse theory for {\em compact} manifolds.

At the beginning of the Eighties, S.P. Novikov~\cite{Nov81b,Nov82,Nov86} identified a class that represents a natural first step from the compact to the
non-compact case, namely manifolds $\hat M$ that are the covering space $p:\hat M\to M$, with discrete group $\Gamma$, of a compact
manifold $M$ and the linear subspace of $C^\infty(\hat M)$ of functions whose differential $df$ is invariant by $\Gamma$ -- a class clearly
modeled on the paradigmatic case of $\bR^n\to\bT^n=\bR^n/\bZ^n$ and the subspace of $C^\infty(\bR^n)$ of pseudoperiodic
functions (see Sec.~\ref{sec:ppf}).
From a more algebraic point of view, this corresponds exactly to extending Morse theory from {\em functions} to {\em closed
  1-forms}. Indeed, while such a function $f$ on $\hat M$ does not generically descend to a function on $M$ (since, for every $x\in M$,
it has different values at the points $f(\gamma(x))=f(x)$ for all $\gamma\in\Gamma$), the {\em exact} 1-form $df$ on $\hat M$
{\em does} descend to a well-defined {\em closed} 1-form $\omega=p_*(df)$ on $M$; inversely, if $\omega$ is a closed 1-form on $M$,
there exist some abelian covering $p:\hat M\to M$ where $p^*\omega$ is exact. In Novikov's terminology, $f$ (or, by abuse of language,
$\omega$) is a {\em multivalued} function on $M$ and so the corresponding Novikov-Morse theory can be also simply thought
as a Morse theory for multivalued functions. 

The Novikov-Morse theory has two main goals~\cite{Nov91}:
\begin{enumerate}
\item To estimate the number of critical points $m_i(\omega)$ with Morse index $i$ for a closed 1-form $\omega$ on $M$.
\item To investigate the topological structure of the leaves of the foliation $\omega=0$ on $M$, i.e. the level sets of the multivalued function $f$.
\end{enumerate}
The present article is a survey of the second point, developed since Eighties mainly by Novikov and his Moscow school.
We refer the reader interested in the first one to the excellent monographs on the subject by M. Farber~\cite{Far04} and A. Pajitnov~\cite{Paj06}.

The survey is structured as follows.
In Section 2 we present the main results on the Novikov general problem of the structure of level sets of multivalued functions on a compact
manifold $M$. One of the main results is that such level sets are quasiperiodic manifolds, namely they are the (generally infinite) union
of a finite number of compact components.
From Section 3, and throughout the rest of the paper, we set $M=\bT^n$. Following V.I. Arnold~\cite{Arn91}, we call
multivalued functions on $\bT^n$ {\em pseudoperiodic} -- note that each pseudoperiodic function is the sum of a linear function
and a periodic one, see Sec.~\ref{sec:ppf}. Arnold's problem, developed jointly with D.A. Panov and I.A. Dynnikov, is studying the
topology of the level lines of a pseudoperiodic map $f:\bT^n\to\bR^{n-1}$ with {\em all} multivalued components
(namely the linear part of this map has full rank).
Section 4 is dedicated to a very rich and important particular case, namely the level lines of a pseudoperiodic map $f:\bT^n\to\bR^{n-1}$
having exactly one single-valued component and all remaining pseudoperiodic components with linear part only. This case was introduced by Novikov,
that extracted the case $n=3$ from solid state physics literature, and studied it jointly with his pupils A. Zorich, I.A. Dynnikov, S. Tsarev,
A. Maltsev and the present author, and more recently with A. Skripchenko. Besides having a surprisingly rich and complicated topological
and geometrical structure, this case is strictly related to the physics of metals and, in particular, it is essential for the theoretical
first-principle prediction of some important experimental data concerning the magnetoresistance of metals. In this section we present the main
analytical and numerical results and some experimental data of the Fifties that was compared with the theory only recently by
the present author thanks to those results.
Finally, in Section 5
we consider a case somehow ``dual'' to the Novikov one in Section 4, namely the case of multivalued maps $f:\bT^n\to\bR^{n-1}$
where {\em all but one} the components are pseudoperiodic functions with linear part only. This case was introduced by A. Zorich
in Nineties as a generalization of the aforementioned problem of Novikov for $n=3$, that corresponds to the study of foliations induced
on a Riemann surface $i:M^2_g\emb\bT^3$ by the 1-form $\omega=i^*B$, where $B$ is a constant 1-form on the torus. Clearly a closed
1-form $\omega$ built this way is very far from being generic. Increasing the dimension of the torus (with respect to the genus
of the surface) makes the 1-form, loosely speaking, closer and closer to being {\em generic}, allowing the use of more powerful techniques.
Even in this case, the asymptotics of the leaves turned out to be topologically very rich. The main contributions to this field are due to Zorich,
M. Kontsevich and A. Avila.
%
%
%
\section{Quasiperiodic manifolds}
\label{sec:Nov}
%
One of the two main goals of the Novikov-Morse theory is studying the structure of the level sets of multivalued functions
or, equivalently, of the leaves of foliations induced by closed 1-forms.
Before getting to the main results of this Section, we list a few important results on this topic 
predating the Morse-Novikov theory. All these results tend to be mostly about how the leaves wind inside
the space rather than on the structure of the leaves; moreover, the authors seem to be unaware of any application of the subject
outside of the field of geometry. 

We will need throughout the article the following definition:
%
\begin{definition}
  The {\em degree of irrationality} of a closed 1-form $\omega$ on a compact manifold $M$ is
  $\irr\omega=\dim_\bQ\langle\int_{\gamma_1}\omega,\dots,\int_{\gamma_r}\omega\rangle_\bQ$, namely
  the dimension over $\bQ$ of the vector space spanned in $\bR$ by $\kappa_i=\int_{\gamma_i}\omega$,
  where the $\gamma_i$ are a base for $H_1(M,\bZ)$ and $r=\dim_\bR H_1(M,\bR)$ is the first Betti number of $M$.
  Clearly the irrationality degree of $\omega$ ranges from 1 to $r$. We say that $\omega$ is {\em rational} if $\irr\omega=1$ and
  {\em fully irrational} if $\irr\omega=\rk H_1(M,\bZ)$.
\end{definition}
The first fundamental result in this field is perhaps the following~\cite{Kol53}:
\begin{theorem}[Kolmogorov (1953)]
  Any closed 1-form on the 2-torus without critical points is smoothly conjugate to a constant 1-form.
\end{theorem}
This result was generalized to $\bT^3$ by Arnold~\cite{Arn92} in 1992 and to any $\bT^n$ by V.V. Kozlov~\cite{Koz07} in 2006.

In Forties Maier proved the following important result:
\begin{theorem}[Maier (1943)]
  \label{thm:mc}
  Any close Morse 1-form $\omega$ on a compact manifold $M$ decomposes it in several connected components of
  two types: {\em periodic components} $\cK_i(M,\omega)$ and {\em minimal components} $\cM_j(M,\omega)$.
  Periodic components are union of compact level sets of $\omega$, while every nonsingular level set in a minimal
  component is dense in it. 
\end{theorem}
\begin{example}
  Consider the case $M=\bT^2$ and $\omega$ constant. Then, when $\irr\omega\leq1$, $\cK(\bT^2,\omega)=\bT^2$
  (and so $\cM(\bT^2,\omega)=\emptyset$) while, when $\irr\omega=2$, $\cM(\bT^2,\omega)=\bT^2$
  (and so $\cK(\bT^2,\omega)=\emptyset$).
\end{example}
Some further important general result was found in Seventies~\cite{Tis70,Ima79}:
\begin{theorem}[Tishler, 1970]
  Suppose that $M$ admits a closed 1-form without singular points. Then $M$ is a fiber bundle over $\bS^1$.
\end{theorem}
\begin{theorem}[Imanishi, 1979]
  \label{thm:Ima}
  If $\omega$ is a closed Morse 1-form on $M^n$, then any open level set of $\omega$ is dense in the minimal component
  that contains it. If $\irr\omega\leq1$, then all its level sets are compact. Moreover, if $\omega$ has no critical points
  of index 1 and $n-1$, then either $\cK(M^n,\omega)=M^n$ (and $\irr\omega\leq1$) or $\cM(M^n,\omega)=M$
  (and $\irr\omega>1$).
\end{theorem}
None of these results aimed at ascertain the structure of the level sets themselves, neither the authors appear aware
of the many connections of this field with other parts of mathematics and other sciences. In particular, the problem
of the magnetoresistance in normal metals (see Section~\ref{subsec:mr}) specifically requires the study of such structures,
which is one of the reasons for Novikov's interest in the problem. Starting from the Eighties, he and his Moscow
school have found several fundamental results that we briefly recall in the rest of this section.
%
\begin{definition}[Novikov, 1982]
  We say that a manifold $L$ is {\em periodic} when it is a $\bZ^k$-covering of a compact manifold for some
  integer $k\geq1$. 
\end{definition}
\begin{theorem}[Novikov 1982, Zorich 1988]
  The non-singular leaves of a closed 1-form $\omega$ without singular points and with $\irr(\omega)=k$
  are periodic manifolds with monodromy group $\bZ^{k-1}$.
\end{theorem}
\begin{proof}
The idea of the proof is the following. Since the (cohomology classes of) rational
1-forms are dense in $H^1(M,\bR)$, we can approximated $\omega$ with a rational one $\omega_0$ as
closely as we please and, moreover, we can choose a basis in $H_1(M,\bR)$ so that $\irr(\omega-\omega_0)=k-1$
and $Ann(\omega_0)=\langle\gamma_2,\dots,\gamma_r\rangle$,
$Ann(\omega)=\langle\gamma_{k+1},\dots,\gamma_r\rangle$. Now, the rational 1-form $\omega_0$
induces a Morse function over the circle $f_\omega:M\to\bS^1$, e.g. by fixing any $x_0\in M$ and setting
$f_0(x)=e^{2\pi i\int_{\gamma_{x_0,x}}\omega_0}$,
where $\gamma_{x_0,x}$ is any smooth curve joining $x_0$ and $x$. 
Indeed the topology of the leaves won't change if we multiply $\omega$ by any non-zero constant and so we
can assume WLOG that all the $\kappa_i$ are integers, so that, by choosing a different path from $x_0$ and $x$,
the integral of $\omega$ will change by an integer. Since $f$ has no critical points, $f:M\to\bS^1$ is a fiber bundle
and we can consider the pull-back bundle $\hat M=exp^*M\to\bR$, where $exp:\bR\to\bS^1$ is the exponential map,
so that $\exp(F(\hat m))=f_0(\pi(\hat m))$, where $\pi:\hat M\to M$ is the $\bZ$-covering canonical projection.
In other words, in $\hat M$ the 1-form $\omega_0$ is {\em exact}: $\pi^*\omega_0=dF$. Since in $\hat M$ we
``unbundled'' the cycle $\gamma_1$, then $\irr\pi^*\omega=k-1$ and, since $\omega$ and $\omega_0$ are
``close enough'', the leaves of $\pi^*\omega$, which are diffeomorphic to those of $\omega$, project along the
trajectories induced by any Riemannian metric onto those of $\pi^*\omega_0$ with monodromy group $\bZ^{k-1}$.
In other words, the leaves of $\omega$ are $\bZ^{k-1}$-coverings of some compact manifold.
\end{proof}
When the closed 1-form $\omega$ has critical points, the situation becomes much more complicated and
leads to the {\em quasiperiodicity} of leaves.
\begin{definition}[Novikov, 1982, 1995]
  Let $\{W_1,\dots,W_p\}$ a finite collection of $k-1$-dimensional manifolds with $(k-1)$-stratified boundary,
  namely with maps $\sigma_j:W_j\to\partial I^{k-1}$ onto the boundary of the unit cube and transversal
  along all sub-cubes of lower dimension. A sequence $j:\bZ^{k-1}\to\{1,\dots,p\}$ is {\em admissible} iff
  we can glue the manifold $W_{j_{n_1,\dots,n_{k-1}}}$ along all neighboring faces of the $(k-1)$-lattice.
  The sequence $j$ is {\em quasiperiodic} iff it is quasiperiodic as a function (in the sense of Weyl\footnote{See~\cite{Che11}
  for a short survey on the several types of almost- and quasi-periodic functions.}). 
  Finally, $j$ is {\em special-quasiperiodic} if $j_{n_1,\dots,n_{k-1}}$ is defined as the
  index of the open set $U_i$ the point $x_0+\sum_{i=1}^{k-1} n_i u_i$ belongs to, where $\{U_1,\dots,U_s\}$ are
  disjoint open sets such that $\bT^{k-1}=\cup_{i=1}^s \overline{U_i}$, $x_0$ a point on the torus and
  $u_1,\dots,u_{k-1}$ vectors such that $x_0+\sum_{i=1}^{k-1} n_i u_i$ never falls on a boundary point of the $U_i$.
\end{definition}
It was conjectured by Novikov in~\cite{Nov82} that non-singular leaves of closed 1-forms with singular points
would be quasiperiodic manifolds. The works of A. Zorich~\cite{Zor88}, Le Tu~\cite{LeT88} and Alanya~\cite{Ala91}
around the end of Eighties, the first in case of small perturbations of rational closed 1-forms and the second two
in the general case, proved the following stronger result:
\begin{theorem}[Zorich 1987, Le Tu 1988, Alanya 1991]
  The non-singular leaves of a closed 1-form $\omega$ are special quasiperiodic manifolds.
\end{theorem}
\begin{proof}
The idea for the general case is the following. Given a closed 1-form $\omega$ with $\irr\omega=k$,
we can always find a basis $\{\omega_1,\dots,\omega_r\}$ of $H^1(M,\bZ)$, dual to a basis
$\{\gamma_1,\dots,\gamma_r\}$ of $H_1(M,\bZ)$, such that $\omega=\sum_{i=1}^k\kappa_i\omega_i$
and $\langle\omega_i,\gamma_j\rangle=0$ for $j\geq k+1$ and $i\geq k$.
Similarly to what done above, the forms $\{\omega_1,\dots,\omega_{k-1}\}$ induce a map $f:M\to\bT^{k-1}$
and therefore a minimal covering $\pi:\hat M\to M$ where $\pi^*\omega$ is exact and a map
$\hat f:\hat M\to\bR^{k-1}$ such that $exp(\hat f(\hat m)) = f(\pi(m))$.
Given any generic point $x_0\in\bR^{k-1}$ and basis vectors $\{e_1,\dots,e_{k-1}\}$, the inverse images
$f^{-1}(I_0^{k-1}+x_0+\sum_{i=1}^{k-1} n_i e_i)$ provide exactly the finite number of manifolds the
leaves of $\omega$ can be built with.
\end{proof}
To the knowledge of the author no progress has been done to date in case of level sets of 2 or more closed 1-forms
since the following conjecture~\cite{Nov91,Nov95}:
%
\begin{conjecture}[Novikov, 1991]
  The generic common level of {\em any number} of closed 1-forms is a quasiperiodic manifold.
\end{conjecture}
A further general result, complementing the result by Imanishi cited above, was found in Nineties by I.A. Melnikova~\cite{Mel95}:
\begin{definition}
  We say that subgroup $\cH\subset H_{n-1}(M^n,\bZ)$ is isotropic if the intersection pairing
  $H_{n-1}(M^n,\bZ)\times H_{n-1}(M^n,\bZ)\to H_{n-2}(M^n,\bZ)$ is identically zero restricted to it.
  We denote by $h(M)$ the largest rank of isotropic subgroups of $H_{n-1}(M^n,\bZ)$.
\end{definition}
\begin{example}
  If $n=2$, $H_1(M^2_g,\bZ)\simeq\bZ^{2g}$ and the largest subgroup over which the intersection product
  is zero has rank $g$.
\end{example}
\begin{theorem}[Melnikova, 1995]
  If $\omega$ is a closed Morse 1-form on $M^n$ and $\irr\omega\geq h(M^n)$, then $\omega$ has a
  non-compact leaf. In particular, a closed Morse 1-form on a Riemann surface $M^2_g$ has a non-compact leaf
  if $\irr\omega> g$.
\end{theorem}
Finally, we mention two interesting results found at the end of Nineties by M. Farber et al.~\cite{FKL98}.
\begin{definition}[Calabi, 1969]
  A closed 1-form is {\em transitive} if, for any $x\in M$ where $\omega(x)\neq0$, there exists a smooth loop
  $\gamma:\bS^1\to M$ based at $p$ such that $\omega(\dot\gamma(t))>0$ for all $t$.
\end{definition}
\begin{theorem}[Farber, Katz and Levine, 1998]
  Let $M$ be a compact manifold such that $H_1(M,\bZ)$ has no torsion and the cup-pairing product
  $$
  H^1(M,\bZ)\times H^1(M,\bZ)\to H^2(M,\bZ)
  $$
  is non-degenerate and let $\omega$ be a closed transitive 1-form with maximal irrationality degree.
  Then $\cM(M,\omega)=M$. 
\end{theorem}
In particular, all non-singular leaves of a fully irrational closed 1-form on $\bT^n$ are dense in the whole space.
%
\section{Pseudoperiodic functions}
\label{sec:ppf}
From this moment on, we will assume throughout the paper that $M^n=\bT^n$.
Following Arnold~\cite{Arn91}, we call {\em pseudoperiodic} the multivalued maps defined on $\bT^n$. 
%
%
By abuse of notation, we will use the same letter to indicate a pseudoperiodic map $\bT^n\to\bR$ and the
corresponding {\em singlevalued} map $\bR^n\to\bR$.
\begin{proposition}
  Any pseudoperiodic function $f$ is  the sum of a linear function $S_f$ with a periodic one $\varepsilon_f$.
\end{proposition}
\begin{proof}
  Let $f:\bT^n\to\bR$ be pseudoperiodic. Then $f$ is a singlevalued function on $\bR^n$ and his differential
  $df$ descends to a closed 1-form $\omega_f$ on $\bT^n$. Let $B_f$ be a constant 1-form such that $[\omega_f]=[B_f]$
  in $H^1(\bT^n,\bR)$. Then $\omega_f-B_f$ is {\em exact}, namely $\omega_f-B_f=d\varepsilon_f$ for some
  {\em singlevalued} function $\epsilon_f:\bT^n\to\bR$. In $\bR^n$, $B_f=dS_f$ for a linear function $S_f$, so
  $f=S_f+\varepsilon_f$ modulo constants.
\end{proof}
Arnold refers to {\em Pseudoperiodic Topology} as the study of level sets of pseudoperiodic maps~\cite{PPT99}
and sets the beginning of this field to the studies on vector fields on tori started by H. Poincar\'e. For example,
in this language the Kolmogorov-Arnold-Kozlov theorem mentioned in the previous section reads as follows:
\begin{theorem}[Kolmogorov, 1953; Arnold, 1992; Kozlov, 2006]
  For any pseudoperiodic function $f$ on $\bT^n$ without critical points, there is a global diffeomorphism
  $F$ of $\bT^n$ such that $\varepsilon_{F^*f}=0$.
\end{theorem}
The following general proposition reflects the fact that clearly every level set of a closed 1-form $\omega=B+d\epsilon$
with $B\neq0$ lies between two planes of the foliation $B=0$ and that, while the topology of a pseudoperiodic
manifold ``at a small scale'' can get quite complicated, it is nevertheless very similar to a plane from a ``global'' point of view:
\begin{proposition}
  Every level set of a non-degenerate pseudoperiodic function $f:\bR^n\to\bR^k$ is contained in a finite radius cylinder
  about the $n-k$-plane $S_f=S_f^{-1}(0)$ and has at least one non-compact component.
\end{proposition}
Since the level set of any non-zero $S_f$ are single hyperplanes, we also have the following:
\begin{proposition}
  For a generic linear function $S$ and periodic function $\varepsilon$ on $\bR^n$, the level sets of the pseudoperiodic function
  $f_\lambda=S+\lambda\varepsilon$ have a single non-compact component for $\lambda$ small enough.
\end{proposition}
%
%
%
%
%
A first non-trivial result on level sets of pseudoperiodic functions that are not necessarily small perturbations of a linear function was found
by Arnold in~\cite{Arn91}, as a first step studying the level {\em lines} of pseudoperiodic maps $f:\bT^n\to\bR^{n-1}$.
\begin{theorem}[Arnold, 1991]
  The level set of a generic pseudoperiodic function on the plane contains precisely one unbounded component.
\end{theorem}
%
%
%
%
%
The case $n>2$ was studied by I.A. Dynnikov in~\cite{Dyn94}:
\begin{theorem}[Dynnikov, 1994]
  Let $a$ be a regular value of a non-degenerate pseudoperiodic map $f:\bT^n\to\bR^{n-1}$ and let $\rho_a$ the radius of
  the smallest cylinder centered at $S_f=S_f^{-1}(a)$ containing $f^{-1}(a)$.
  Then there is an odd number of non-compacts component of $f^{-1}(a)$ and each one is a finite deformation of $S_f$.
  Moreover, if the direction $S_f$ is fully irrational, $\sup_{a\in\bR}\rho_a$ is finite and the number of
  connected components in $f^{-1}(a)$ does not depend on $a$.
 %
\end{theorem}
%
%
Unlike what happens for $n=2$, even in the generic case the number of non-compact components of the level sets of
a pseudoperiodic map $f:\bT^n\to\bR^{n-1}$ can be any positive odd number, as an example of D.A. Panov~\cite{Pan96,Pan99} shows:
\begin{theorem}[Panov, 1996]
  For every $n\geq3$ there is an open set (in the space $\bR^n\times C^\infty(\bT^n)$ of $n$-vectors and $n$-periodic functions)
  of pseudoperiodic maps $f:\bT^n\to\bR^{n-1}$ having level sets with $k=2m+1$ non-compact lines for any integer $m\geq0$.
\end{theorem}
\begin{figure}
  \centering
  \includegraphics[width=6cm]{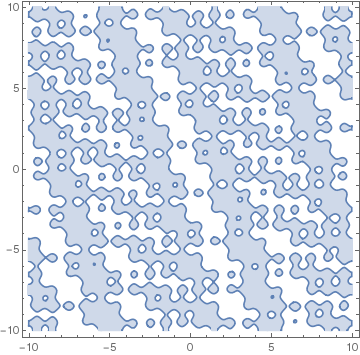}\hskip.75cm\includegraphics[width=6cm]{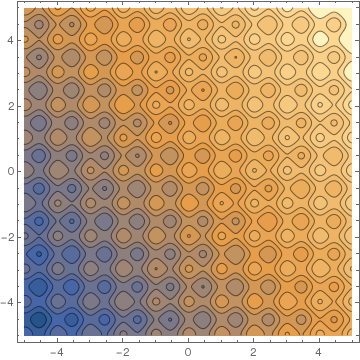}
  \caption{%
    \small
    (left) Sections of the periodic function $f(x)=\sum_{i=1}^5\cos(x^i)$ with a 2-plane spanned by $v=(1,0,.2,.3,.15)$ and $w=(0,1,.05,.1,.4)$. 
    (right) Level sets of the pseudoperiodic function $F(x,y)=y+\sqrt{2}\,x+\cos(2\pi x)+\cos(2\pi y)$ 
  }
  \label{fig:pp}
\end{figure}
\noindent
{\bf Open questions:} what can be said about the number and structure of the level sets of a pseudoperiodic map $\bT^n\to\bR^k$?
%
\section{Quasiperiodic functions on the plane}
\label{sec:qpf}
%
It turns out that non-generic pseudoperiodic maps can be topologically very rich too. The case
when $f=(\varepsilon,S_1,\dots,S_{n-2}):\bT^n\to\bR^{n-1}$, where $\varepsilon$ is singlevalued and the $S_i$ are purely linear
pseudoperiodic functions, was introduced by Novikov in 1982~\cite{Nov82} and leads to the field of
quasiperiodic functions. Indeed, the image in $\bR^n$ of the line sets of $f$ coincide with the line sets
of the restriction of $\varepsilon$ to affine 2-planes defined by the equations $\{S_1=a_1,\dots,S_{n-2}=a_{n-2}\}$,
namely (see the definition below) the line sets of a quasiperiodic function on the plane with $n$ quasiperiods.

{\em Quasiperiodic functions} started appearing naturally in the fields of analysis, geometry and dynamical systems since the end
of the XIX century (see~\cite{DN05} for several related examples and references). They were first introduced in literature by the
Latvian mathematician P. Bohl~\cite{Bohl93},
in the context of the theory of differential equations, and by the French astronomer E. Esclangon~\cite{Esc04},
who introduced the terminology, although they become widely known to the mathematical community only in Twenties through the works of
H. Bohr~\cite{Bohr26} and A. Besikovich~\cite{Bes26}, who further extended this class of functions to {\em almost periodic} functions:
\begin{definition}
  The set $AP(\bR^n)$ of {\em almost periodic} functions over $\bR^n$ is the closure in $C_b(\bR^n,\bC)$, the set of bounded
  continuous complex functions over $\bR^n$, of the set of trigonometric polynomials with respect to the supremum norm.
  For every $f\in AP(\bR^n)$, we define its {\em frequency module} $M(f)$ as the  $\bZ$-module generated by all ``frequencies vectors''
  $\nu\in\bR^n$ such that
  $$
  \lim_{T\to\infty}\frac{1}{(2T)^k}\int_{[-T,T]^n}e^{-i\langle\nu,x\rangle}\overline{f(x)}dx\neq0.
  $$
  We say that $f$ is {\em quasiperiodic} if $M(f)$ is finitely generated and we call the generators {\em quasiperiods} of $f$.
  Equivalently, we say that $f:\bR^n\to\bC$ is quasiperiodic with $m>n$ quasiperiods if
  $$
  f=\varphi\circ\pi_m\circ\ell
  $$
  for some $\varphi\in C^0(\bT^m,\bC)$ and some affine embedding $\ell:\bR^n\to\bR^m$,
  namely if $f$ is the restriction to a $n$-plane of a periodic function on $\bR^m$.
\end{definition}
\begin{example}
  Consider the function $f(x)=\cos(2\pi x)+\cos(\sqrt{2}\,2\pi x)$. Clearly $f$ is a non-periodic almost periodic function.
  Its frequency module is generated over $\bZ$ by 1 and $1/\sqrt{2}$, so $f$ is quasiperiodic with two quasiperiods.
  For example, $f$ can be seen as the restriction of $\varphi(x,y)=\cos(2\pi x)+\cos(2\pi y)$ to the affine line $y=\sqrt{2}x$. 
\end{example}
Note that, if $f'=\varphi\circ\pi_m\circ\ell'$ with $\ell'$ parallel to $\ell$, then, for some $c\in\bR^n$, we have that
$f'(x)=f(x+c)$ for all $x\in\bR^n$. Although in general we cannot find a $c$ such that $f(x+c)=f(x)$ for all $x$, we can
get as close as we please to this condition: 
%
\begin{theorem}[Bohr, 1926]
  A function $f\in C_b(\bR^n,\bC)$ is almost periodic iff there exists a relatively dense set of $\varepsilon$-{\em almost periods}
  for every $\varepsilon>0$, namely for every $\varepsilon>0$ there is a relatively dense set $T_\varepsilon\subset\bR^n$ such that
  $$
  |f(x+a)-f(x)|<\varepsilon\;\hbox{ for all } a\in T_\varepsilon.
  $$
\end{theorem}
%
The main general results on the topological structure of level sets of quasiperiodic functions are due to S.M. Gusein-Zade~\cite{GZ89,GZ99}
and have to do with the topological invariants for the level sets $V_c=f^{-1}(c)$ and lower sets $M_c=f^{-1}((-\infty,c])$
of a generic quasiperiodic function.

As mentioned in the introduction, for both compact and periodic manifolds
are defined fundamental topological numbers such the Betti numbers and the Euler characteristic. Moreover,
in case of the level and lower sets of a Morse function on a compact manifold, these numbers would be piecewise
constant functions of the parameter $c$. On the contrary, such numbers are all infinite for a generic non-compact
manifold or, correspondingly, for a the level and lower levels of a generic function on a non-compact manifold.
In case of quasiperiodic manifolds, though, it is still possible to give a meaning to such quantities by considering
the corresponding {\em averaged} quantities in the following way, that we write in case of the Euler characteristic:
\begin{definition}
  Let $R>0$, $B_R\subset\bR^n$ the ball of radius $R$ centered at the origin and $f:\bR^n\to\bR$ a Morse
  function. Then the density of critical points of $f$ of order $k$ is
  $$
  N^k_f(c) = \lim_{r\to\infty}\frac{N^k_f(B_R\cap V_c)}{Vol(B_R)}
  $$
  and similarly are defined the density of Euler characteristic $\chi_f(c)$ and of Betti numbers $b^k_f(c)$
  (resp. $\tilde\chi_f(c)$ and $\tilde b^k_f(c)$ ) of $M_c$ (resp. $V_c)$. 
\end{definition}
%
%
\begin{theorem}[Gusein-Zade, 1999]
  For any analytic (resp. generic smooth) quasiperiodic function $f$
  which is the restriction to a $n$-plane of a $m$-periodic function, the density $N^k_f(c)$ is a well
  defined piecewise analytic (resp. piecewise smooth) function of $c$. Moreover, in the neighborhood of
  every singular point $c_0$, $f$ has an asymptotic expansion of the form
  $$
  \sum_\alpha\sum_{q=0}^{m-n-1} A_{i,\alpha}|c-c_0|^\alpha \ln^q|c-c_0|\,,
  $$
  where $\alpha$ ranges on some finite segment of an arithmetic progression of rational numbers
  with difference equal to 1.
\end{theorem}
\begin{theorem}[Gusein-Zade, 1999]
  For all smooth  (resp. almost all smooth) quasiperiodic functions $f$,
  $$
  \chi_f(c)=\sum_{k\geq0}(-1)^kN^k_f(c).
  $$
  Moreover, $\tilde\chi_f(c)=0$ for $n$ even and $\tilde\chi_f(c)=2\chi_f(c)$ for $n$ odd.
  Finally, $\chi_f(c)=\sum_{k\geq0}(-1)^kb^k_f(c)$ when the $b^k_f(c)$ are all well defined.
\end{theorem}
It is still an open question whether or not the density of Betti numbers is well defined under the same conditions
holding for the Euler characteristic:
\begin{conjecture}[Gusein-Zade, 1999]
  For all analytical (resp. almost all smooth) quasiperiodic functions, the Betti numbers densities are well defined.
\end{conjecture}
This conjecture was proved in 2000 by A.I. Esterov~\cite{Est00} in the smooth case for the Betti numbers densities of $M_c$:
\begin{theorem}[Esterov, 2000]
  For almost all smooth quasiperiodic functions, the Betti numbers densities $b^k_f(c)$ are well defined and
  depend continuously on $c$.
\end{theorem}
%

%
\subsection{Quasiperiodic functions in $\bR^2$ with 3 quasiperiods}
The case $n=3$ is the one that Novikov extracted from solid state physics literature (see~\cite{Nov82}) and that attracted
his interest in the general problem of levels of quasiperiodic functions. 
In this section we present in detail the physics of the model from which the problem originated, in which played a relevant
role the Kharkov-Moscow school of solid state physics (among which I. Lifshitz, M. Azbel and M. Kaganov), the analytical
results, found mainly by Novikov and his topology school in Moscow, and the numerical results by the present author.
\subsubsection{Physics}
\label{subsec:mr}
The solid state physics basic model for a normal metal is a lattice of ions $\Gamma\subset\bR^3$ with rank 3 surrounded
by a gas of free electrons.
In order to get manageable equations for this model, electrons are considered independent from each other, so that it's enough
to study the problem for a single one. The quantum Hamiltonian operator $H=-\Delta+V$ on $L^2(\bR^3)$, where $\Delta$
is the Laplacian and $V$ the potential energy, corresponding to the ions lattice commutes with the action of $\Gamma$ on $\bR^3$,
so that all of its eigenstates $\psi_n$ satisfy a relation
$$
\psi_n(x+\gamma)=\exp(2\pi i \langle p,\gamma\rangle)\psi_n(x)
$$
for some covector $p\in(\bR^3)^*$ clearly defined modulo $\Gamma^*$, the dual lattice of $\Gamma$ (Bloch theorem, e.g. see Ch. 8 in~\cite{AM76}).
The corresponding eigenvalues $\varepsilon_n(p)$ of $H$ (energy bands) therefore are functions over $(\bR^3)^*/\Gamma^*\simeq\bT^3$.
The electrons responsible for the conductivity of the metal correspond to a particular $n=n_0$ and a particular energy $E_0$.
The Riemann surface $\{\varepsilon(p)=E_0\}\emb\bT^3$, where we set $\varepsilon=\varepsilon_{n_0}$, is called
{\em Fermi Surface} (FS) of the metal and the labels $p$ are called {\em quasimomenta} of the electrons. 
\begin{figure}
  \centering
  \includegraphics[width=9cm]{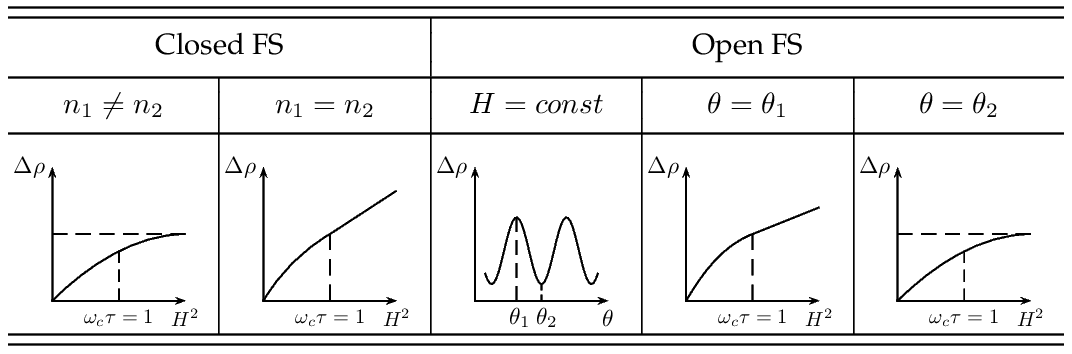}
  \caption{%
    \footnotesize
    Behavior of $\bm\rho=\bm\sigma^{-1}$ in metals with closed and open FS\cite{LK80}. 
    (Closed) $\bm\rho$ is isotropic and saturates unless the density of electrons and holes 
    coincide, in which case $\rho\sim\|\bH\|^2$. 
    (Open) $\bm\rho$ is highly anisotropic and it shows qualitatively different behavior 
    in minima and maxima (resp. $\theta_2$ and $\theta_1$ in the picture): 
    in maxima $\rho\sim H^2$, in minima it saturates ($\theta$ is the angle 
    between $\bH$ and the crystallographic axis).
  }
\label{fig:mrgraph}
\end{figure}

Even in the free electrons approximation, studying the Schrodinger equation after adding a {\em weak} (enough not to
perturb the lattice) external magnetic field  is known to be a very hard problem already in two dimensions (see Novikov's
results in~\cite{Nov81} and~\cite{Nov83}). In particular, no exact magnetic analogue of Bloch waves is known when the magnetic
flux is irrational.
This hard obstacle led physicists to introduce several alternate models to study phenomena involving external fields in metals.
In late Fifties, I. Lifshitz, M.Ya. Azbel and M.I. Kaganov~\cite{LAK56} (from the Kharkov-Moscow school of solid state physics),
in order to explain some anysotropies found experimentally in the angular dependence of magnetoresistance on the direction
of an external magnetic field, proposed using a {\em semiclassical approximation} for the system.

Under this approximation, electrons are treated as ``classical particles'' but with $\bT^3$, rather than $\bR^3$,
as phase space. Under a {\em constant} magnetic field $B=B^idp_i$, the equations of motions are given by the classical ones:
$$
\dot p_i = \{p_i,\varepsilon(p)\}_B\,,
$$
where
$$
\{p_i,p_j\}_B = \epsilon_{ijk}B^k
$$
is the ``magnetic'' Poisson structure on $\bT^3$ and $\epsilon_{ijk}$ is the Levi-Civita tensor.
This dynamical system turns out to be quite non-trivial and related to the problem of level sets
of quasiperiodic functions with 3 quasiperiods. The bracket $\{,\}_B$, indeed, has a Casimir,
namely a function commuting with all other functions, represented by the ``function'' $S(p)=B^kp_k$,
but clearly $S$ on the phase space $\bT^3$ is {\em multivalued}. Notice that very few
Hamiltonian and Poissonian systems with multivalued first integrals have been studied by the
dynamical systems community to date, probably because very few of them originate from classical
mechanics. Notice also that the orbits of the quasimomenta of the electrons under these equations of motion
are given by the intersection of the FS by planes perpendicular to the magnetic field, namely they
are the level sets of a multivalued map $f=(\varepsilon,S):\bT^3\to\bR^2$ with the first
component singlevalued.

The relevance of the geometry and topology of the FS in physical phenomena is well known since Thirties,
when Justi and Scheffers showed evidences that the Fermi Surface $G$ of Gold is {\em open}~\cite{JS37} --
namely, in more formal terms, that the rank of the rings homomorphism $i_*:H_1(G,\bZ)\to H_1(\bT^3,\bZ)$
induced by the embedding $i:G\emb\bT^3$ is larger than 0 (in fact, it is maximal).
The model introduced in~\cite{LAK56} suggested that the behavior of 
magnetoresistance in monocrystals at low temperatures in high magnetic fields
would also be quite sensitive to the FS topology: as the intensity of the magnetic field $B$ grows, 
the magnetoresistance saturates isotropically to an asymptotic value if the orbits are all closed while it grows 
quadratically with $B$ if there are open orbits; moreover, in this last case the 
magnetoresistance is not isotropic and the conducibility tensor $\bm\sigma$ has rank 1.

%
\begin{figure}
  \centering
  \includegraphics[width=5cm]{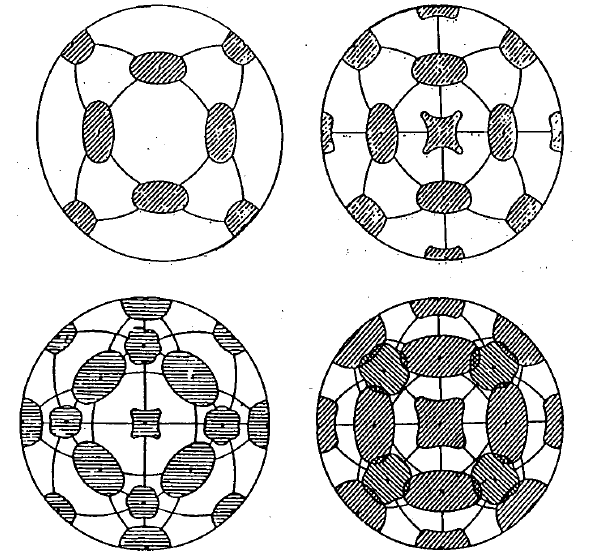}\hskip.75cm\includegraphics[width=5cm]{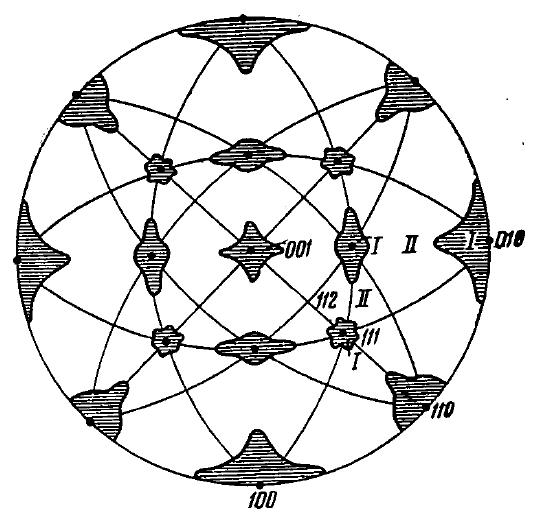}
  \caption{%
    \small
    (left) Four qualitative sketches for the SM of a metal suggested by Lifshitz and Peschanskii in~\cite{LK60}.
    (right) Experimental data about magnetoresistance in gold monocrystals obtained by Yu.P. Gaidukov~\cite{Ga60}.
  }
\label{fig:Ga60}
\end{figure}


To these first theoretical predictions of the Lifshitz model (see Fig.~\ref{fig:mrgraph}) 
followed many other works studying magnetoresistance from both the 
experimental~\cite{Pip57,AG59,Ga60,AG60,AGLP61,AG62a,AG62b,AG62c,AG63,AKM64}
and theoretical~\cite{LAK56,LAK57,LAK73,LP59,LP60,Cha57,Cha60,Sho60,Sho62,Pip89} point 
of view; in particular, Stereographic Maps (SM) were experimentally built by plotting 
in a stereographic projection all magnetic field directions $B$ in which a quadratic rise 
of $\bm\sigma$ was observed (see Fig.\ref{fig:Ga60}). In those times the interest
on these magnetoresistance effects was due mainly to their utility as a tool to determine 
FS properties rather than as phenomena in their own right, and in particular SM maps provided
information about FS topology: e.g. clearly if $\bm\sigma$ grows quadratically for some direction
of $B$ then the FS must be open, and further analysis can lead to discover the directions of the openings.

Between Fifties and Seventies SM were experimentally found for about thirty metals
including Gold~\cite{Ga60}, Silver~\cite{AG62c}, Copper~\cite{Pip57,KRBK66},
Lead~\cite{AG62a}, Cadmium~\cite{AG63}, Zinc~\cite{GK66}, Thallium~\cite{AG63},
Gallium~\cite{RM62} and Tin~\cite{AG62b}.
Despite the theoretical efforts, though, no way was found to generate them with first-principles 
calculations and therefore no accurate direct verification of the Lifshitz model is 
available to date, except for the qualitative sketches by Lifshitz and Peschanskii~\cite{LP60}
(see Fig.\ref{fig:Ga60}); in particular, it was not known till now how closely the semiclassical 
model is able to reproduce these complex experimental data and, therefore, whether or not
further purely quanto-mechanical corrections to the model are needed.
As the magnetoresistance methods were replaced by newer and more accurate tools to
study FS, no new SM was experimentally produced and the problem was eventually abandoned.

The results on topological structure of the level lines of $f=(\epsilon,S)$ illustrated in
next section will shed full light on the relation between the direction of the magnetic
field and the direction of the non-compact trajectories, when they arise.
\subsubsection{Topology}
\label{subsec:q3m}
The main problem for the case $n=3$ can be formulated as follows: given a triply periodic function
$\varepsilon:\bR^3\to\bR$ and a constant 1-form $B$, determine when do open level
lines arise as intersections of a level surface $\varepsilon_c=\varepsilon^{-1}(c)$ with a level plane of $B$ and,
in case they do, find their asymptotics. We call simply $B$-sections such level lines. Note that $\varepsilon$
descends in $\bT^3$ to a smooth function and $B$ to a closed 1-form. By abuse of notation, we will use the
same symbol for the corresponding objects in $\bR^3$ and $\bT^3$. Since the $B$-sections only depend on
the {\em direction} of $B$, we will sometimes consider, by a further abuse of notation, $B\in\RPt$. It will be
clear by the context which object we mean by $B$.

The following definition is fundamental for this problem: 
\begin{definition}
  We call {\em topological rank} of an embedding $i:M^2_g\to\bT^3$ of a Riemann surface into the the 3-torus
  the rank of the induced ring homomorphism $ i_*:H_1(M^2_g,\bZ)\to H_1(\bT^3,\bZ)$.
\end{definition}
\begin{remark}
  Since $H_1(M^2_g,\bZ)\simeq\bZ^{2g}$ and  $H_1(\bT^3,\bZ)\simeq\bZ^{3}$,
  in order to have an embedding with maximal rank we must have $g\geq2$
  (see Fig.~\ref{fig:FSlevels} for examples).
\end{remark}
%
%

With respect to foliations induced on $M^2_g$ by general closed 1-forms, here we have a further structure due to the
embedding of $M^2_g$ inside $\bT^3$ and to the fact that the level lines are all planar in the universal covering.
This special situation allows us the following decomposition of $M^2_g$.
Let $C_i$ be the components of $\varepsilon_c$ filled by leaves homologous to zero in $\bT^3$. Then $M^2_g\setminus\cup C_i$
is the union of periodic and minimal components $\cK_j$ and $\cM_k$ whose boundaries are all loops homotopic to zero.
\begin{definition}
  The {\em genus} of the components $\cK_j$ and $\cM_k$ of $M^2_g$ is the genus of the corresponding manifold
  obtained by quotienting to a single point each boundary component. 
\end{definition}
The following observations by Dynnikov~\cite{Dyn96} are fundamental to understand the topological structure behind this problem:
\begin{theorem}[Dynnikov, 1996]
  If a component $\cK_j$ or $\cM_k$ of $\varepsilon_c$ has {\em topological rank} 2, then there is no component with higher rank and
  all components with topological rank 2 of $\varepsilon_{c'}$, $c'\in\bR$, share the same integral indivisible homology class in $H_2(\bT^3,\bZ)$ {\em modulo sign}.
\end{theorem}
\begin{theorem}[Dynnikov, 1996]
  If a component $\cK_j$ or $\cM_k$ of $\varepsilon_c$ has {\em genus} $g_0>1$, then there is a neighborhood $U$ of $c$
  such that all components of $\varepsilon_{c'}$ have genus smaller than $g_0$ for all $c'\in U$.
\end{theorem}
%
%
\begin{definition}
  We denote by $\cF_B(M^2_g)$ the foliation induced by $B$ on $M^2_g\emb\bT^3$ and we say that $\cF_B(M^2_g)$ is:
  \begin{enumerate}
  \item {\em trivial} if both the genus and the topological rank of all components $\cK_j$ are not larger than 1 (in this case no $\cM_k$ arises);
  \item {\em integrable} if the genus of all its components $\cK_j$ and $\cM_k$ is not larger than 1 and there is at least one component with topological rank 2;
  \item {\em chaotic} if some component $\cK_j$ or $\cM_k$ has genus larger than 1.
  \end{enumerate}
  In the integrable case, the indivisible 2-cycle $\ell\in H_2(\bT^3,\bZ)$ common, modulo sign, to all rank-2 components
  is called the {\em soul} of $\cF_B(M^2_g)$.
\end{definition}
\begin{remark}
  Every non-zero integer 2-cycle in $\bT^3$ determines a {\em rational direction} in its universal cover $\bR^3$. By abuse of notation,
  therefore, {\em throughout this section we often consider $\ell$ as an element of $\QPt$ rather than $H_2(\bT^3,\bZ)$.}
\end{remark}
The structure of open $B$-sections in the integrable case is very simple~\cite{Dyn92}:
\begin{proposition}[Dynnikov, 1992]
  Let $N\emb\bT^3$ be a surface embedded with topological rank 2 and let $\ell=[N]\in H_2(\bT^3,\bZ)$. Then
  the open $B$-sections of $N$ are strongly asymptotic to a straight line with direction $\ell\times B$.
\end{proposition}
Once a function $\varepsilon$ has been fixed, since only the direction of $B$ determines the foliation, the ``phase space'' of this problem 
is $[a,b]\times\RPt$, where $[a,b]=\varepsilon(\bT^3)$. Correspondingly, there are two natural ways to tackle this problem: either fix a level
for $\varepsilon$ and study what happens by tilting $B$, or fix $B$ and study what happens by changing the
``energy level''. The first approach was taken by Zorich and resulted in the following theorem (see Fig.~\ref{fig:FSlevels}), based on a clever
elementary topological argument~\cite{Zor84}:
\begin{theorem}[Zorich, 1984]
  \label{thm:rational}
  Consider an embedding $i:M^2_g\to\bT^3$ and let $B$ be a {\em rational} constant 1-form on $\bT^3$ such that
  $\omega=i^*B$ is a Morse closed 1-form on $M^2_g$. Then $\cF_{B}(M^2_g)$ is either trivial or integrable and $\cF_{B'}(M^2_g)$ is,
  correspondingly, trivial or integrable for all $B'$ close enough to $B$.
\end{theorem}
\begin{proof}
  Every rational closed 1-form on $M^2_g$ is the differential of a map $f_\omega:M^2_g\to\bS^1$. This is a Morse map
  by hypothesis and so decomposes $M^2_g$ in elementary cobordisms, the only non-trivial one being pants,
  {\em whose boundary loops are $B$-sections}.
  Since, by construction, two of the three boundaries of such pants lie on the same 2-torus embedded in $\bT^3$,
  it follows easily that at least one of the three must be homologous, and so homotopic, to zero in $\bT^3$.
  Hence, the genus of all periodic components $\cK_i$ (if any) is equal to 1 (by Thm~\ref{thm:Ima} there are
  no minimal components). Under small enough perturbations, closed leaves non homotopic to zero might become open
  (namely minimal components might arise) but the leaves homotopic to zero will stay so and so all periodic
  and minimal components will still have genus 1, namely $\cF_{B}(M^2_g)$ is either trivial or integrable.
\end{proof}
\begin{figure}
  \centering
  \includegraphics[width=5cm]{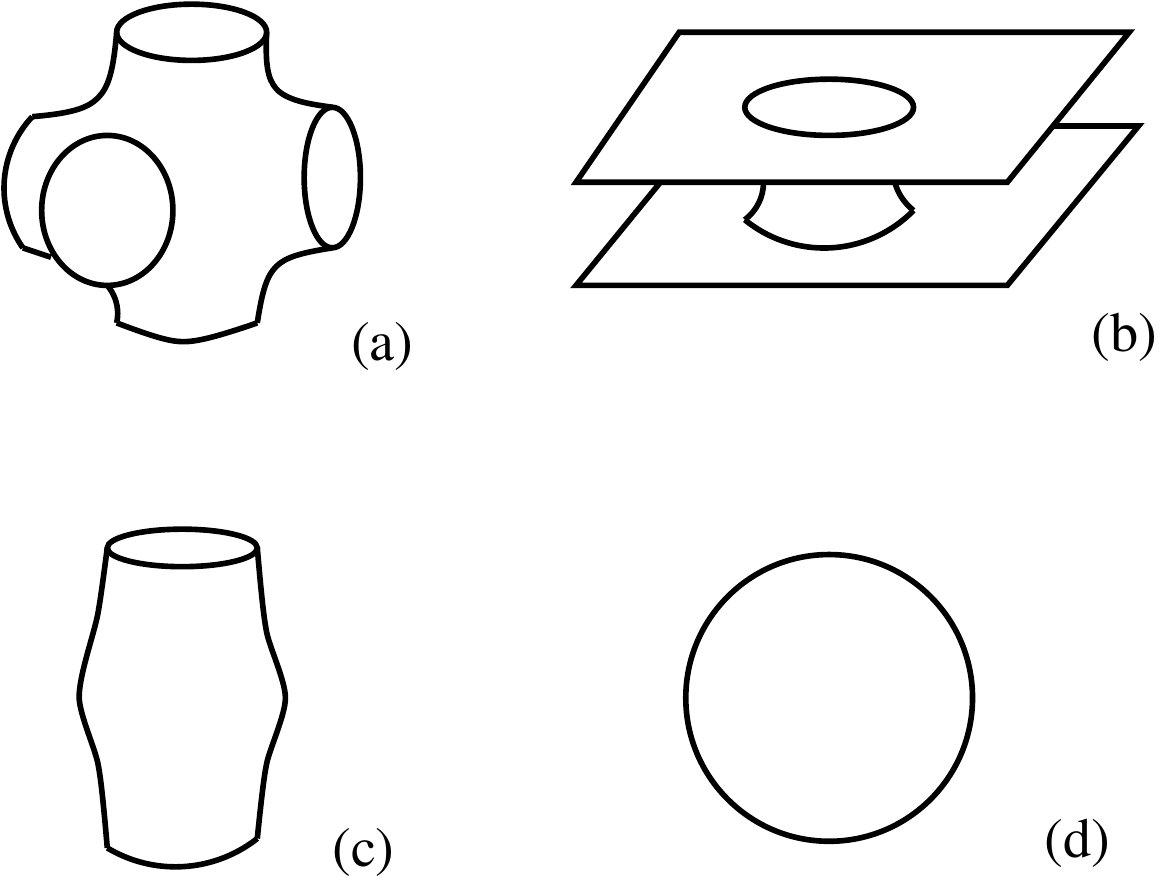}\hskip.75cm\includegraphics[width=5cm]{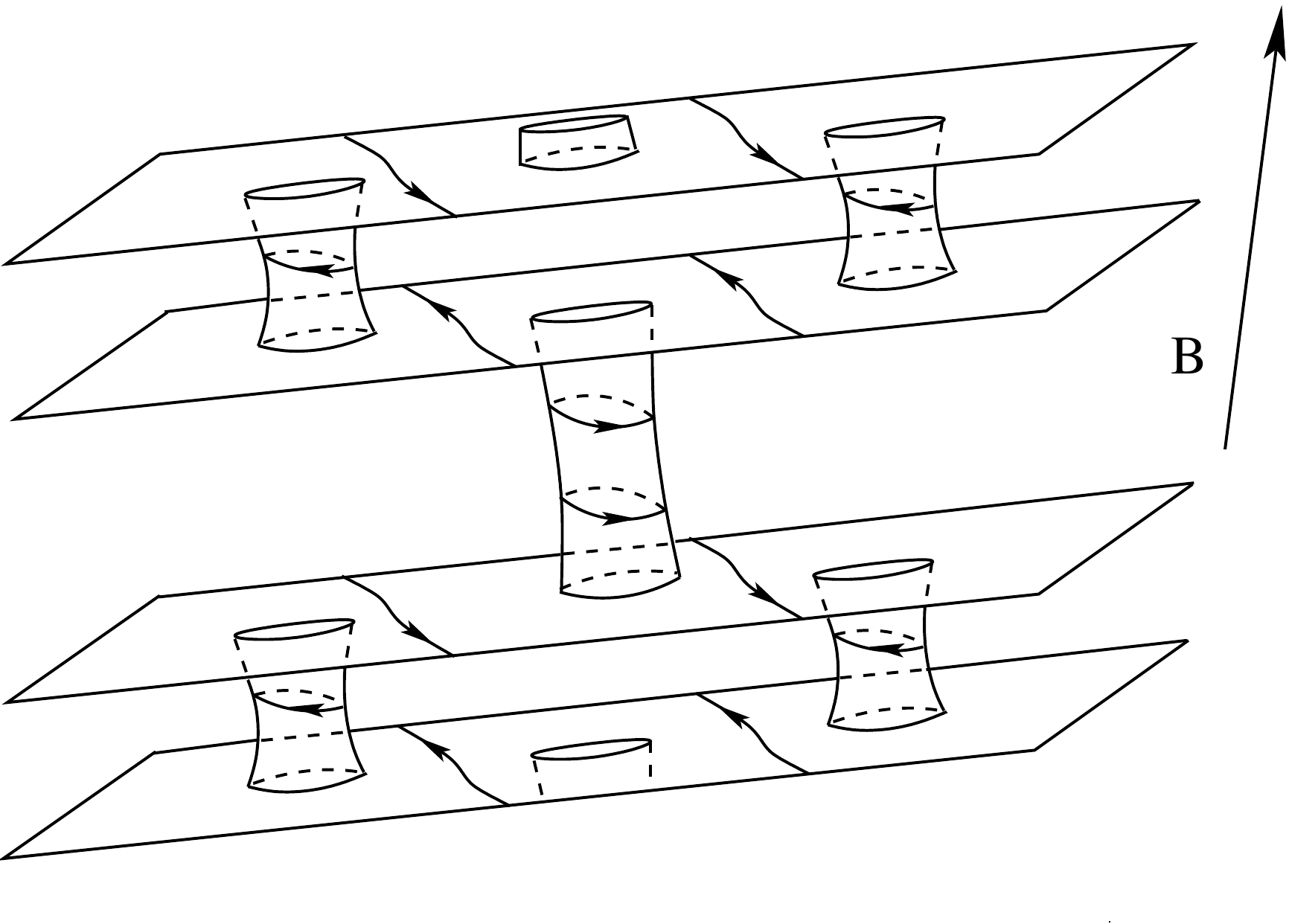}
  \caption{%
    \footnotesize
    (left) Surfaces embedded with rank 0, 1, 2 and 3 in $\bT^3$.
    (right) A genus 4 surface embedded with rank 3 in $\bT^3$ and some of its closed and open $B$-sections for some $B$ nearly vertical.
  }
\label{fig:FSlevels}
\end{figure}
%
The second approach was taken by I.A. Dynnikov, who proved, in a series of articles~\cite{Dyn92,Dyn96,Dyn99}
and with much more sophisticated arguments, the results below.
\begin{theorem}[Dynnikov, 1999]
  Let $\varepsilon:\bT^3\to\bR$ be a Morse function. Then there are functions
  $e_1,e_2:\real{\mathrm P}^2\rightarrow\real$ such that
  \begin{itemize}
  \item
    $e_1(B)\leqslant e_2(B)$ for all $B\in\real{\mathrm P}^2$;
  \item
    $\cF_B(\varepsilon_c)$ is trivial if and only if $c\notin[e_1(B),e_2(B)]$;
  \item
    if $e_1(B)<e_2(B)$, then $\cF_B(\varepsilon_c)$ is integrable for all $c\in[e_1(B),e_2(B)]$,
    and its soul $\ell_B(\varepsilon_c)$ is independent of $c$.
  \end{itemize}
\end{theorem}
\begin{definition}
  Given a surface $M^2_g$ embedded in $\bT^3$ and a $\ell\in\rational\mathrm{P}^2$, 
  we call {\em stability zone} for $M^2_g$ relative to $\ell$ the set 
  $$
  \zone_{\ell}(M^2_g)=\{B\in\real\mathrm{P}^2\;;\;\ell_{M^2_g,B}=\ell\}\,.
  $$
  We denote by $\cD(M^2_g)$ the union of all stability zones and we call {\em stereographic map} of $M^2_g$ the map
  $$
  \ell_{M^2_g}:\cD(M^2_g)\to\rational\mathrm{P}^2
  $$
  associating the soul to each $B$ that has one. Finally, we denote by $\cE(M^2_g)$ the set of directions for which
  $\cF_B(M^2_g)$ is chaotic.
\end{definition}
\begin{proposition}[Dynnikov, 1999]
  Let $M$ and $N$ two disjoint embedded surfaces in $\bT^3$. Then $\ell_M$ and $\ell_N$ are compatible
  on all directions on which they are both defined.
  \end{proposition}
\begin{definition}
  Given a function $f:\bT^3\to\bR$, and a $\ell\in\rational\mathrm{P}^2$, we call {\em stability zone}
  for $f$ relative to $\ell$ the set 
  $$
  \zone_{\ell}(f)=\cup_{c}\cD_\ell(f_c)
  $$
  and we denote by $\cD(f)$ the union of all stability zones. We call {\em stereographic map} of $f$ the map
  $$
  \ell_{f}:\cD(f)\to\rational\mathrm{P}^2
  $$
  associating the soul to each $B$ that has one for some $c\in\bR$. Finally, we set
  $$
  \cE(f)=\cup_{c}\cE(f_c).
  $$
\end{definition}
%
%
%
\begin{theorem}[Dynnikov, 1999]
  \label{p1}
  For a generic surface $M^2_g\subset\bT^3$, the sets $\zone_{M^2_g,\ell}$ are disjoint closed domains with piece-wise
  smooth boundary. The set $\cE(M^2_g)$ is disjoint from $\rational\mathrm P^2$ and has zero measure.
  The set of directions $B$ with trivial $B$-sections is open.
\end{theorem}
\noindent
Examples of such stability zones are shown in Fig.~\ref{fig:nm}.
%
\begin{theorem}[Dynnikov, 1999]
  \label{p3}
  Stability zones of a generic $f\in C^\infty(\bT^3)$ are closed domains with piecewise smooth boundary.
  If $\ell\ne\ell'$ then $\mathcal D_{\ell}(f)$ and $\mathcal D_{\ell'}(f)$ can only have
  intersections at the boundary and the number of their common points is at most countable.
  Finally, either the whole $\real{\mathrm P}^2$ is covered by a single generalized stability zone or 
  the number of zones is countably infinite and the set $\cE(f)$ 
  is uncountable.
\end{theorem}
It follows that the interior of $\mathcal D(f)$ is an open everywhere dense subset of
$\real\mathrm P^2$ and its complement $\overline{\mathcal E(f)}$
has the form of a two-dimensional cut out fractal set~\cite{DeL05a,DeL03b}.
Examples of such maps $\ell_f$ are shown in Fig.~\ref{fig:cos3D}. Note that several properties
about the structure of the stability zones are not well understood yet:
\begin{conjecture}[Dynnikov, 1999]
  \label{conj:size}
  The area of a stability region $\zone_{\ell}(M^2_g)$ does not exceed $C/\|\ell\|^3$ for some constant $C$
  that depends only on $M^2_g$. The sets $\mathcal D_{\ell}(f)$ are connected and simply connected.
\end{conjecture}
\noindent
\renewcommand\labelitemi{$\spadesuit$}
{\bf Open Questions:}
\begin{itemize}
\item Is the number of sides of a stability zone finite?
\item Are stability zones convex?
\item If not, can two stability zones meet in more than one point?
\item Does a generic surface $M^2_g$ have only finitely many stability zones?
\item Does any of these theorems still hold, possibly in a weaker form, if we consider,
  rather than constant 1-forms, closed 1-forms with zeros?
\item What can be said if we replace $\bT^3$ by some other compact 3-manifold (e.g. $M^2_g\times\bS^1$) 
  and $B$ with a closed 1-form without zeros?
\end{itemize}
%

%
The following theorem connects $\mathcal E(f)$ with the image of the stereographic map $\ell_f$~\cite{DeL03b}:
\begin{theorem}[De Leo, 2003]
  \label{p4}
  If there is more than one generalized stability zone, then the closure of $\mathcal E(f)$ coincides with
  the set of accumulation points of the set $\ell_f(\cD(f))$.
\end{theorem}
It is plausible, but still unknown, that $\mathcal E(f)$ has always zero measure.
The following stronger conjecture was proposed in~\cite{MN03}:
\begin{conjecture}[Novikov, 2003]
  \label{conj:zm}
  Whenever $\cE(f)$ is non-empty, the Hausdorff dimension~\cite{Fal90} of $\cE(f)$
  is strictly between 1 and 2 for every $f$.
\end{conjecture}
Another important but still not completely understood problem is the structure of chaotic level sets.
In~\cite{Dyn99} Dynnikov presented an abstracted sophisticated construction that showed the
existence of such sections whose closure is contained in a minimal component of topological rank larger
than 2.  In~\cite{DD07}, in case of a polyhedral surface (see next section), Dynnikov and the present author
provided for the first time a concrete case of a chaotic section on a simple surface. Shortly afterwards,
Dynnikov started using the Interval Exchange theory tools to study chaotic sections,
provided new examples with similar but more generic surfaces and proposed the following conjecture~\cite{Dyn08}:
\begin{conjecture}[Dynnikov, 2008]
  If $\cF_B(M^2_g)$ is chaotic, almost all the $B$-sections of $M^2_g$ consist in a single connected curve.
\end{conjecture}
In 2013 A. Skripchenko~\cite{Skr13} supported the conjecture by building an example of surface
for which there is at least a direction on which this is true. More recently, though, Dynnikov and
Skripchenko~\cite{DS15} constructed examples of surfaces $M^2_g$ such that $\cF_B(M^2_g)$ is chaotic
but on each $B$-section of $M^2_g$ lie infinitely many components. Moreover, in this case, typical
connected components do have an asymptotic direction, showing in particular that the corresponding
foliation on $M^2_g$ is not uniquely ergodic. 

The most recent result on this subject is by A. Avila, P. Hubert and S. Skripchenko, that studied
the case of the \mucube\ (namely the regular skew apeirohedron $\{4,6|4\}$, see~\cite{Cox37} and
next section) introduced by Dynnikov and the present author in~\cite{DD07}, and proved
the following important particular result~\cite{AHS16}:
%
\begin{theorem}[Avila, Hubert, Skripchenko, 2016]
  For almost all $B$ with $\cF_B(\mucube)$ chaotic, every $B$-section has an asymptotic
  direction with a diffusion rate strictly between $1/2$ and $1$. 
\end{theorem}
\subsubsection{Numerical Analysis}
\label{sec:num}
No algorithm is known to generate analytically, in general, the stability zones relative to a given surface or function.
Given a direction $B$ inside a stability zone $\cD_{\ell}(M^2_g)$, in principle the analytical boundaries of $\cD_{\ell}(M^2_g)$
can be found by looking at the cylinders $C_i$ of closed orbits separating the pairs of genus-1 rank components.

Let us discuss in some detail the simplest non-trivial case, namely when $M^2_g$ is a surface of genus 3.
Note that, numerically, we can work only with directions in $\QPt$ but this is not a significant restriction
for this problem because $\QPt$ is dense in $\cD(f)$ for every $f$. For practical reasons, we choose to represent 
a the direction $B$ one of the two integer indivisible vectors in its equivalence class.
Since $B$ is rational, all $B$-sections will be compact, some homologous to zero  in $\bT^3$ and some not. 
In the integrable case, there will be exactly two genus-1 components $\cK_{1,2}$ filled by $B$-sections non homologous
to zero and they will be separated by exactly two cylinders $C_{1,2}$ of sections homologous to zero. The heights
of the two cylinders is non-zero in some open topological discs $D_{1,2}$ containing $B$ and the corresponding
stability zone is given by $\cD=D_1\cap D_2$. The zone $\cD$ is subdivided in open sets $E_j$ corresponding to
different pairs of critical points at the bases of the cylinders. The equation of the side of each $E_j$ is given
by $\langle B,p_2-p_1\rangle=0$, where $p_{1,2}$ are the critical points at the bases of the cylinder that
determines the breaking of the $P_{1,2}$ (see Fig.~\ref{fig:zone} for an example).
\begin{figure}
  \centering
  \includegraphics[width=13cm]{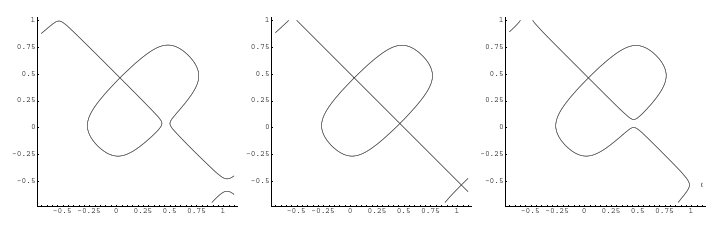}
  \caption{%
    \footnotesize
    A change of cylinder inside a stability zone of $\cD(\fc)$. On the left is shown a critical leaf at the base of a cylinder, the
    critical point is $p_1\simeq(0.035,0.463,0.25)$. At the opposite base lies the critical point $p=p_4+(0,0,1)$. The middle
    picture shows what happens at the boundary between the two stability zones of cylinders, namely the point $p_1$ has a
    saddle connection with $p_2$. The picture on the right shows the base of the new cylinder. At one base still lies
    the point $p_1$ but at the opposite one now lies $p_4+(1,1,0)$.
  }
\label{fig:zone}
\end{figure}
%

Using this procedure, Dynnikov~\cite{Dyn96} started the numerical study of this problem by finding ten stability zones
of the stereographic map of one of the simplest trigonometric polynomials with level sets of topological rank equal to 3,
the function
$$
\fc(x,y,z)=\cos(2\pi x)+\cos(2\pi y)+\cos(2\pi z)\,,
$$
whose regular level sets $\fc_c=\fc^{-1}(c)$ are either spheres (for $c<-1$ and $c>1$) or genus-3 surfaces (for $-1<c<1$) .
Note that the surface $\fc_c$ is a translate of $\fc_{-c}$ and therefore, in this case, $e_1(B)=-e_2(B)$ for all $B\in\RPt$,
so that $\cD(\fc)$, in principle equal to the union of all the $\cD(\fc_c)$, in this case is simply equal to $\cD(\fc_0)$. 
Unfortunately, though, the procedure above is hard to implement in a computer language.

In order to produce a reasonably good approximation of a SM $\ell_{M^2_g}$ we decided, therefore, to rather
select a lattice of rational directions covering homogeneously $\RPt$ or some proper subset of it
(in our computations so far we used a number of points ranging from $10^4$ to $10^8$) and
evaluating the soul (when defined) for each of such directions. In concrete, for each $B$ on the lattice we evaluate numerically
the four critical sections, namely the $B$-sections through the points where the plane perpendicular to $B$
is tangent to $M^2_g$. When $\cF_B(M^2_g)$ is integrable, all critical points at the bases of the two cylinders are
homoclinic saddles with one critical loop homologous to zero and all four such loops are homologous in $M^2_g$
to the same $\gamma\in H_1(M^2_g,\bZ)$. Then we evaluate the homology class $\gamma$ and 
the rank-2 subring $Ann(\gamma)\subset H_1(M^2_g,\bZ)$ of all (cohomology classes of) loops having zero
intersection numbers with $\gamma$. The image of the annihilator of $\gamma$, $Ann(\gamma)$, into
$H_1(\bT^3,\bZ)$ is a rank-2 subring $R$ that, in turn, determines a rank-1 subring $L\subset H_2(\bT^3,\bZ)$
as the set of all 2-cycles whose intersection number with all elements of $R$ is zero. The soul of $\cF_B(M^2_g)$
is, modulo sign, any of the two indivisible elements of $L$. 
%
\begin{figure}
  \centering
  \includegraphics[width=5.4cm]{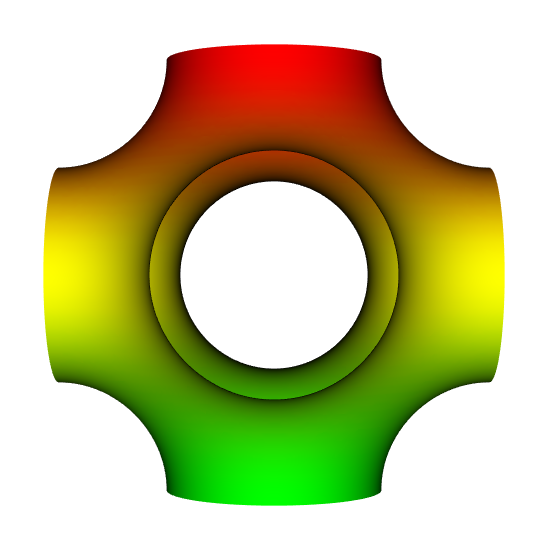}\hskip.76cm\includegraphics[width=5.4cm]{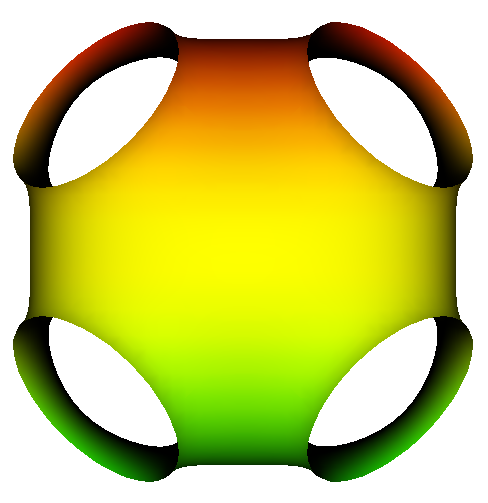}\\
  \vspace{.3cm}
  \includegraphics[width=5.4cm]{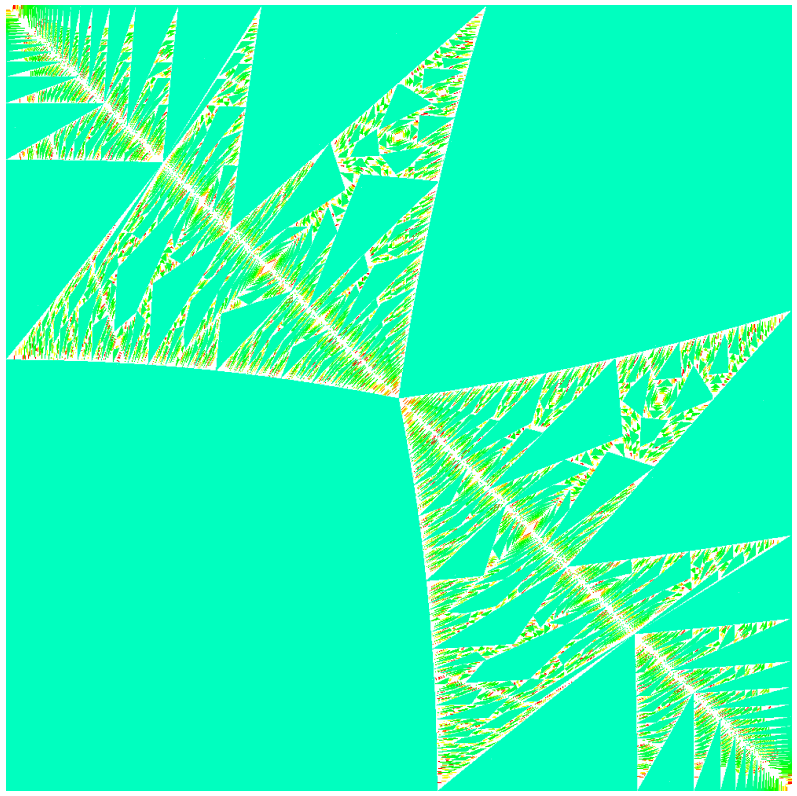}\hskip.76cm\includegraphics[width=5.4cm]{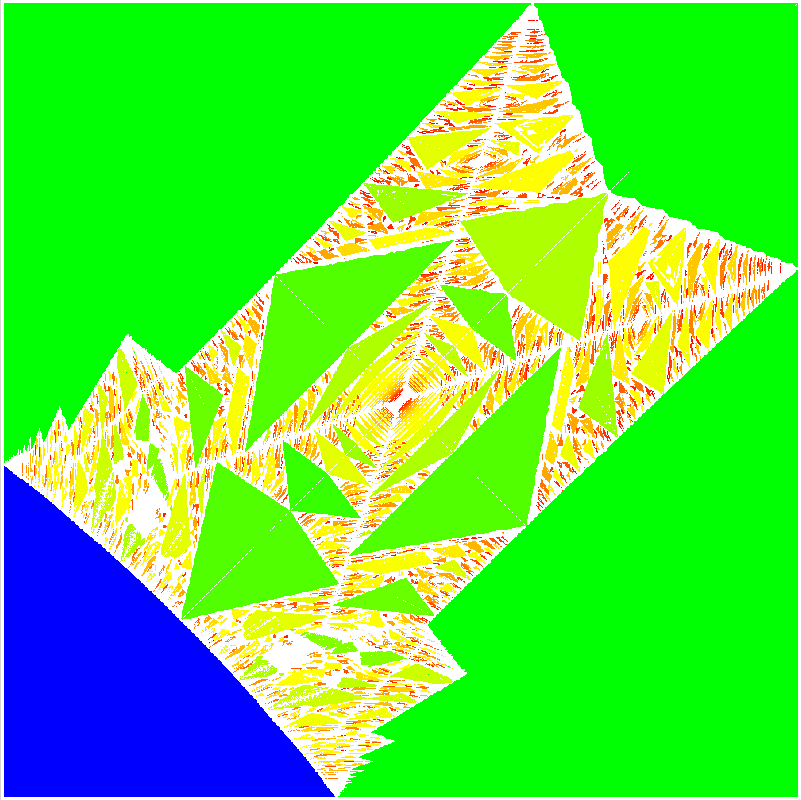}\\
  \vspace{.3cm}
  \includegraphics[width=5.4cm]{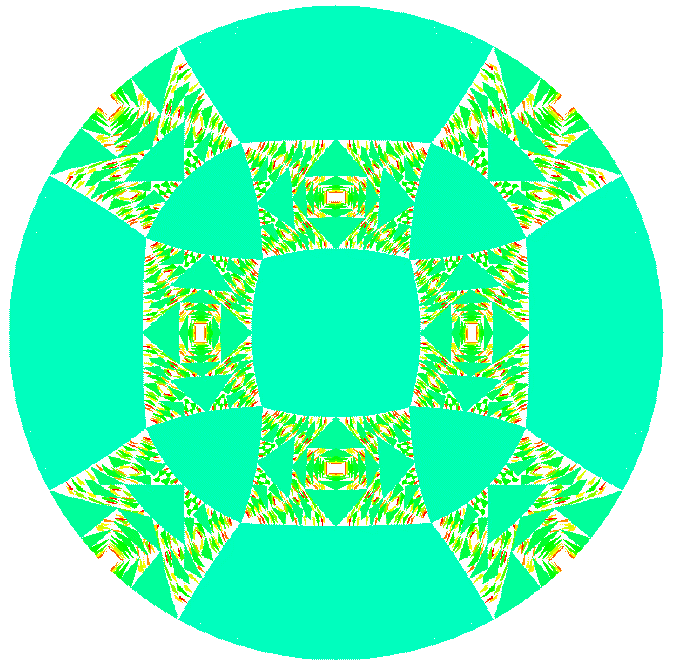}\hskip.76cm\includegraphics[width=5.4cm]{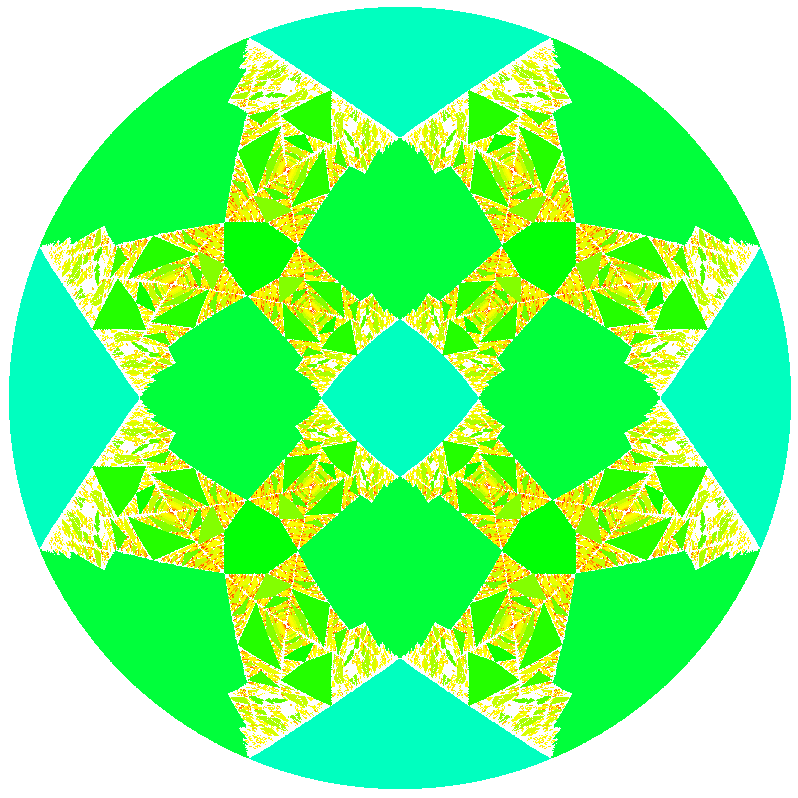}\\
  \caption{%
    \footnotesize
    (left) From top to bottom: the surface $\fc_0$, the SM $\cD(\fc)$ in the square $[0,1]^2$ of the chart $B_z=1$ and in the whole $\RPt$.
    (right) From top to bottom: the surface $\fd_0$, the SM $\cD(\fd)$ in the square $[0,1]^2$ of the chart $B_z=1$ and in the whole $\RPt$.
  }
\label{fig:cos3D}
\end{figure}

We implemented this algorithm in NTC~\cite{DeL03a} (nts.sf.net), an open source C++ library we built on top of the well-known
3D computer graphics, image processing, and visualization open source C++ library VTK (www.vtk.org)
by W. Schroeder, K. Martin and B. Lorense~\cite{SML06}. In order to check its reliability, we successfully tested NTC's results
on the SM of eight surfaces, the zero level of the function $\fc$ above and several level sets of a piecewise quadratic function --
all these surfaces are genus-3, are embedded with topological rank 3 and are invariant by permutations of the coordinate
axes. Their symmetry is cubic, namely their Brillouin zone is the unit cube. 
In Fig.~\ref{fig:cos3D} (left) we show $\fc_0$, the whole SM $\cD(\fc)$ and a detail of it in the region $[0,1]^2$ in the chart $B_z=1$.
A rough numerical evaluation of its box dimension gives an estimate of about 1.83, in agreement with Novikov's Conjecture~\ref{conj:zm}.

We later generalized NTC to genus-4 surfaces and used it to evaluate the SM of FS of noble metals (see next section).
Recently, in order to see which kind of SM arise from functions of higher genus and less symmetry, we started the study of
the SM of the function
$$
\fd(x,y,z)=\cos(2\pi  x)\cos(2\pi  y)+\cos(2\pi  y)\cos(2\pi  z)+\cos(2\pi  z)\cos(2\pi  x)\,,
$$
whose regular level sets $\fd_c$ are either spheres (for $c<-1$ and $c>0$) or genus-4 surfaces (for $-1<c<0$).
Each of the genus-4 level sets has topological rank 4. Note also that $\fd$, besides being invariant by integer translations
along the coordinate axes, is invariant with respect to translations by $1/2$ along the cube diagonals, namely it
has a bcc invariance.
In Fig.~\ref{fig:cos3D} (right) we show the whole SM $\cD(\fd)$ in its first Brillouin zone, the SM $\cD(\fd)$ and
a detail of it in the region $[0,1]^2$ in the chart $B_z=1$. A rough numerical evaluation of its box dimension of about 1.69,
again in agreement with Novikov's Conjecture~\ref{conj:zm}.

In order to have a look at simpler SM of functions and, at the same time, to decrease substantially the time to generate, in general,
SM of surfaces, we generalized in~\cite{DeL06} the main results of Zorich and Dynnikov to polyhedral surfaces. The most noticeable
difference with the smooth case is that in the polyhedral case saddle points with index smaller than -1 (e.g. monkey saddles) are generic. 
In Fig.~\ref{fig:cos3D} we show the results for three noteworthy ones: two of the three regular skew apeirohedra (see~\cite{Cox37}),
the muoctahedron (with Schl\"afli symbol $\{6,4|4\}$) and the \mucube\ (with Schl\"afli symbol $\{4,6|4\}$), and the
cubic polyhedra $P_1$, namely the only cubic polyhedron with all {\em screw} vertices, similarly to the \mucube\ being
the only cubic polyhedron with all monkey-saddle vertices (see~\cite{GSS02}). 
Each one of these surfaces is triply periodic and enjoys the following property: their interior is identical, modulo translations,
to their exterior, so that their open $B$-sections arise for them for every $B$ and therefore their SM coincides with the SM of some
function. 
A rough numerical evaluation of their box dimension gives an estimate of 1.69, 1.72 and 1.68, in agreement with Novikov's Conjecture~\ref{conj:zm}.
\begin{figure}
  \centering
  \includegraphics[width=4.25cm]{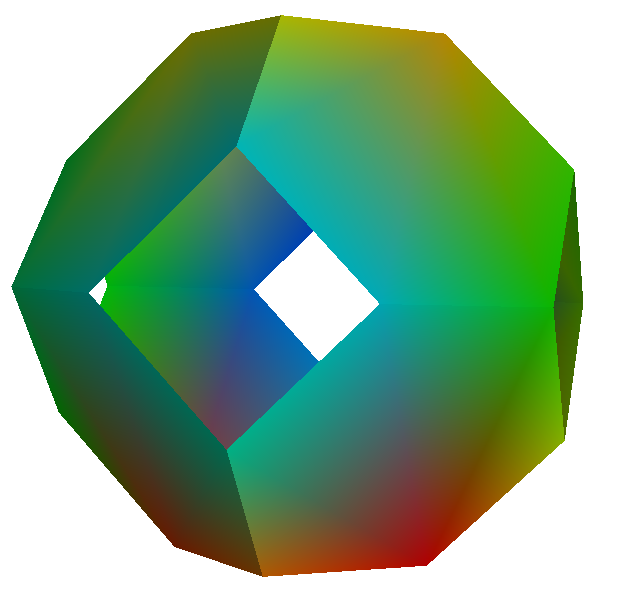}\hskip.40cm\includegraphics[width=4.25cm]{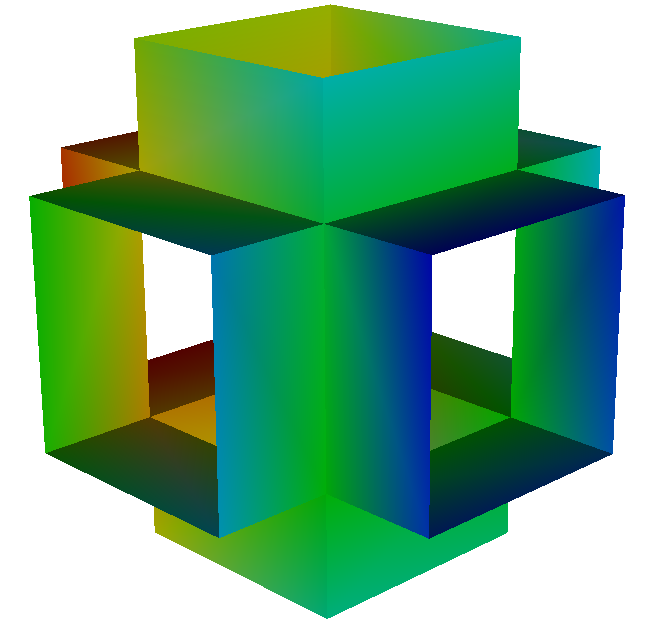}\hskip.40cm\includegraphics[width=4.25cm]{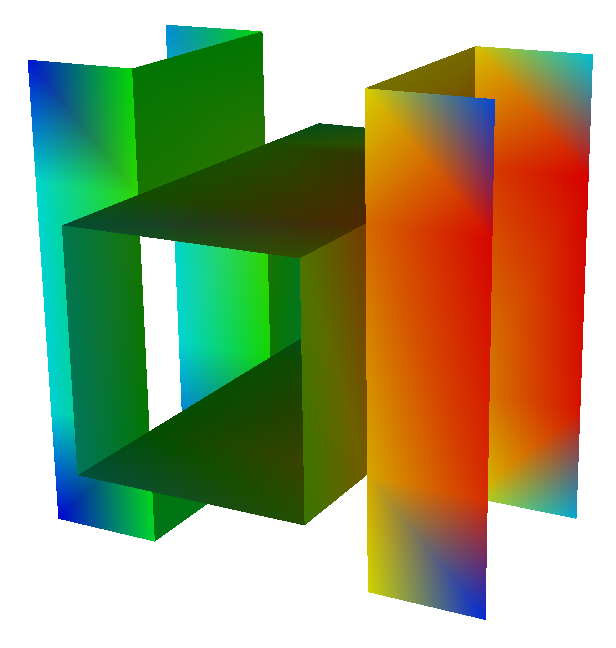}\\
  \vspace{.3cm}
  \includegraphics[width=4.25cm]{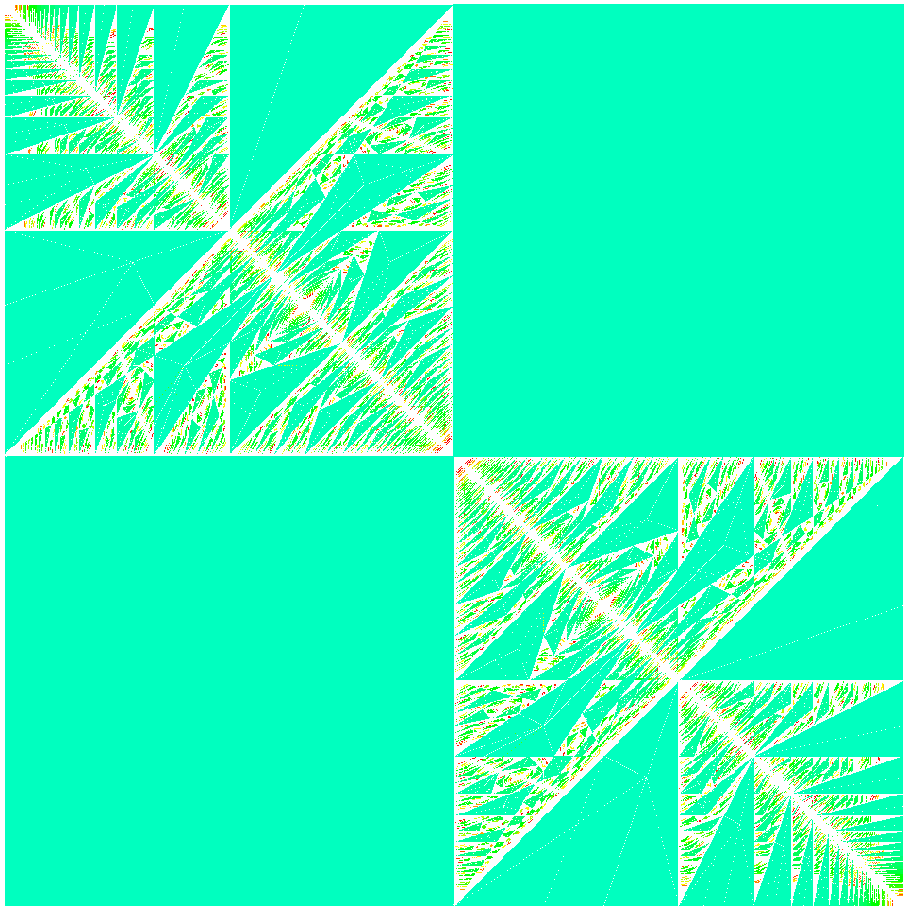}\hskip.40cm\includegraphics[width=4.25cm]{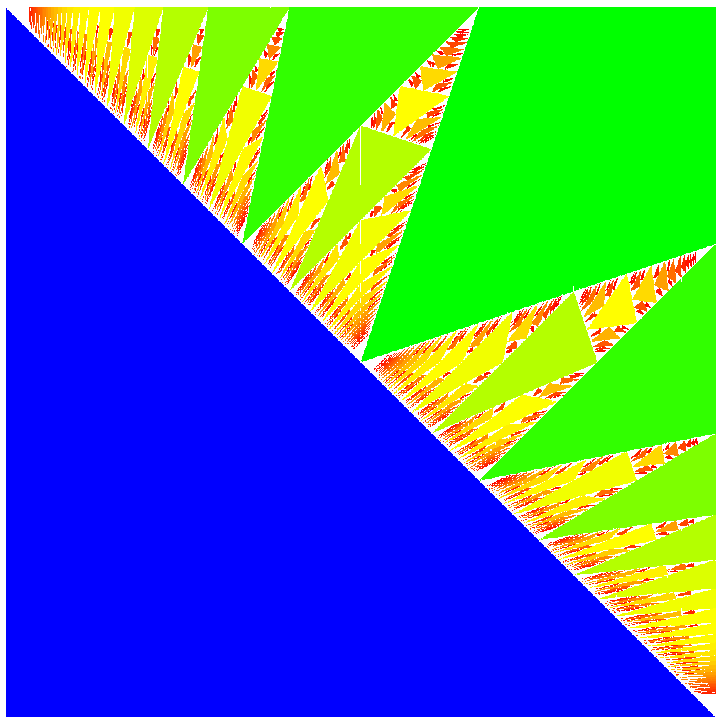}\hskip.40cm\includegraphics[width=4.25cm]{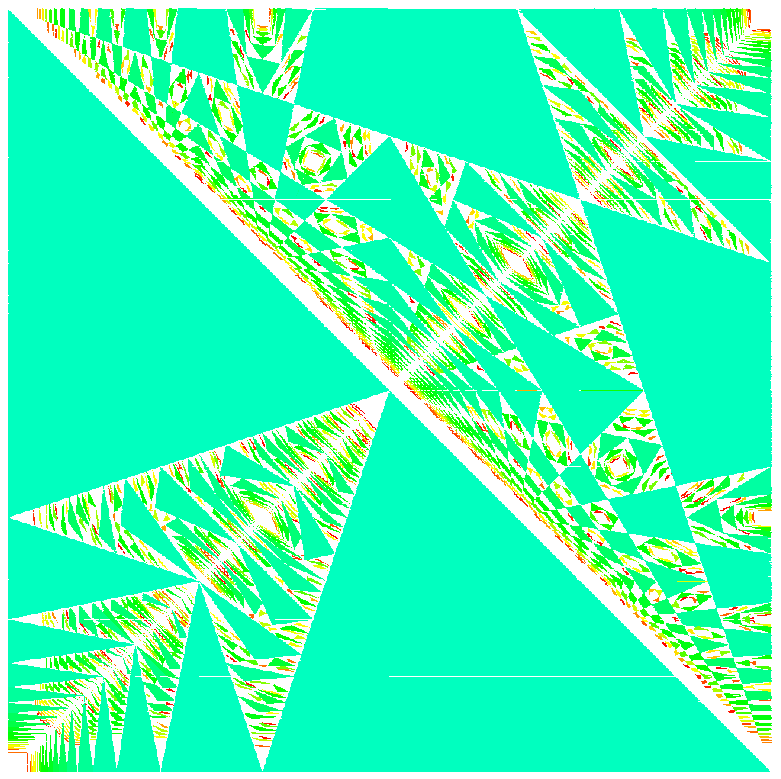}\\
  \vspace{.3cm}
  \includegraphics[width=4.25cm]{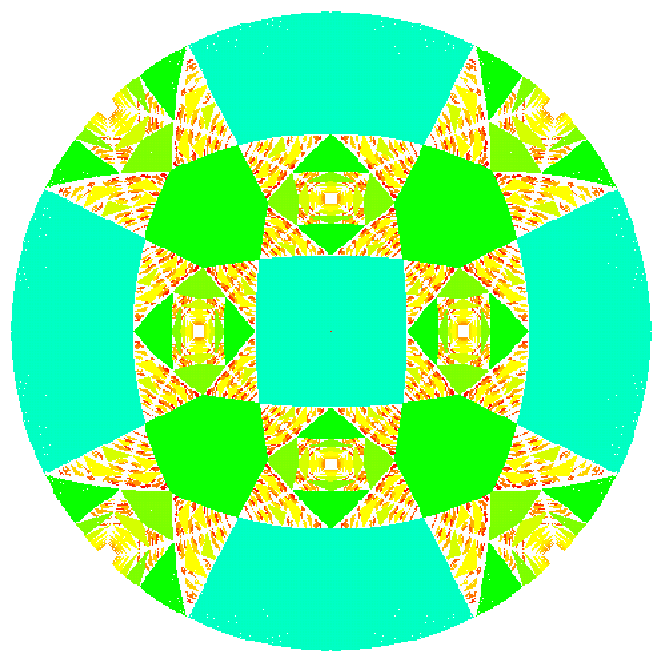}\hskip.40cm\includegraphics[width=4.25cm]{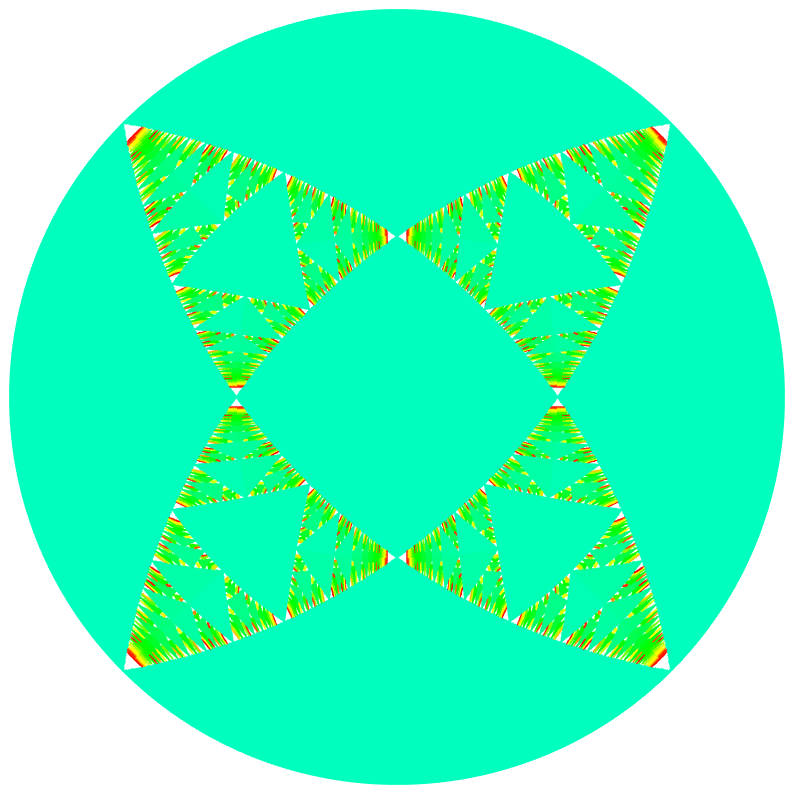}\hskip.40cm\includegraphics[width=4.25cm]{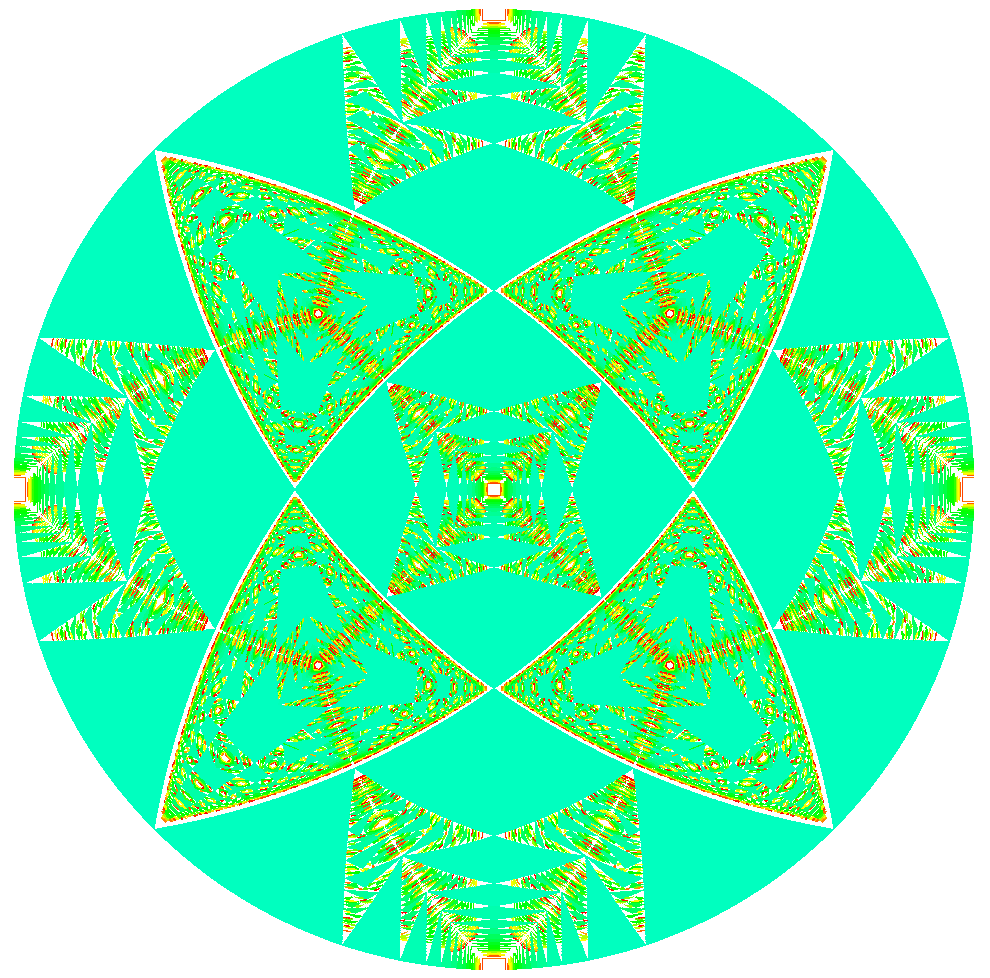}\\
  \caption{%
    \small
    (left) From top to bottom: the $\mu$-octahedron, its SM $\cD(\mu\hbox{-octahedron})$ in the square $[0,1]^2$ of the chart $B_z=1$ and in the whole $\RPt$.
    (center) From top to bottom: the $\mu$-cube, its SM $\cD(\mucube)$ in the square $[0,1]^2$ of the chart $B_z=1$ and in the whole $\RPt$.
    (right) From top to bottom: the $P_1$ polyhedron, its SM $\cD(P_1)$ in the square $[0,1]^2$ of the chart $B_z=1$ and in the whole $\RPt$.
  }
\label{fig:pol}
\end{figure}

The strategy of looking for the most possible elementary cases proved to be successful and it produced a breakthrough. In case
of the \mucube, indeed, we were able to find an explicit algorithm producing the non-trivial part of $\cD(\mucube)$ and, in
a joint work with Dynnikov~\cite{DD09}, we were able to prove several fundamental properties about the structure of $\cD(\mucube)$:
\begin{theorem}[De Leo, Dynnikov, 2009]
  The following properties hold for the $\mucube$:
  \begin{enumerate}
  \item The largest SZ of its SM are the (closed) square $Q_z$ with sides $[0:1:1]$, $[1:0:1]$, $[0:-1:1]$ and $[-1:0:1]$
    and the two other ones $Q_x$ and $Q_y$ obtained, respectively, by switching the $z$ axis with the $x$ axis and with the $y$ axis.
    The set $\RPt\setminus(Q_x\cup Q_y\cup Q_z)$ is the disjoint union of four open triangles $T_i$ and the restriction
    of $\cD(\mucube)$ to each $T_i$ is a Levitt-Yoccoz gasket.
  \item $\cE(\mucube)$ has measure zero.
  \item there exist four constants $A$, $B$, $C$, $D$ such that
    $$                                                                                                                                                                                                                                   
    \frac{A}{\|\ell\|^{\frac{3}{2}}}\leqslant p_\ell                                                                                                                                                                                     
    \leqslant\frac{B}{\|\ell\|}\;,\;\;\frac{C}{\|\ell\|^3}\leqslant a_\ell \leqslant\frac{D}{\|\ell\|^3}                                                                                                                                 
    $$
    for all souls of $\cD(\mucube)$.
  \item Let $\ell_n$ any recursive ordering of the souls of $\cD(\mucube)$. Then
    $$
    2\log_3^2(1+2n)+1\leqslant \|\ell_n\|^2\leqslant3(1+2n)^{2\log_3\alpha}\,,
    $$
    where $\alpha$ is the Tribonacci constant.
    \end{enumerate}
  \end{theorem}
The Levitt-Yoccoz gasket (see Fig.~\ref{fig:dc}) is a self-projective parabolic fractal set (see~\cite{DL15b} for details on its story and geometry)
generated by the following recursive algorithm: given a triangle in $\RPt$ with vertices $[e_1]$, $[e_2]$ and $[e_3]$, subtract from it the
inscribed one with vertices $[e_1+e_2]$, $[e_2+e_3]$ and $[e_3+e_1]$ and proceed recursively with the three triangles left.
The second and third statement proves Conjecture~\ref{conj:size} and partially Conjecture~\ref{conj:zm} for the $\mucube$.
The fourth led the present author to the study of attractors of general self-projective semigroups based on the asymptotics
of the norms of corresponding semigroups of linear transformations~\cite{DL15}. 
\begin{figure}
  \centering
    \includegraphics[width=7cm,clip=true,trim=80 250 20 100]{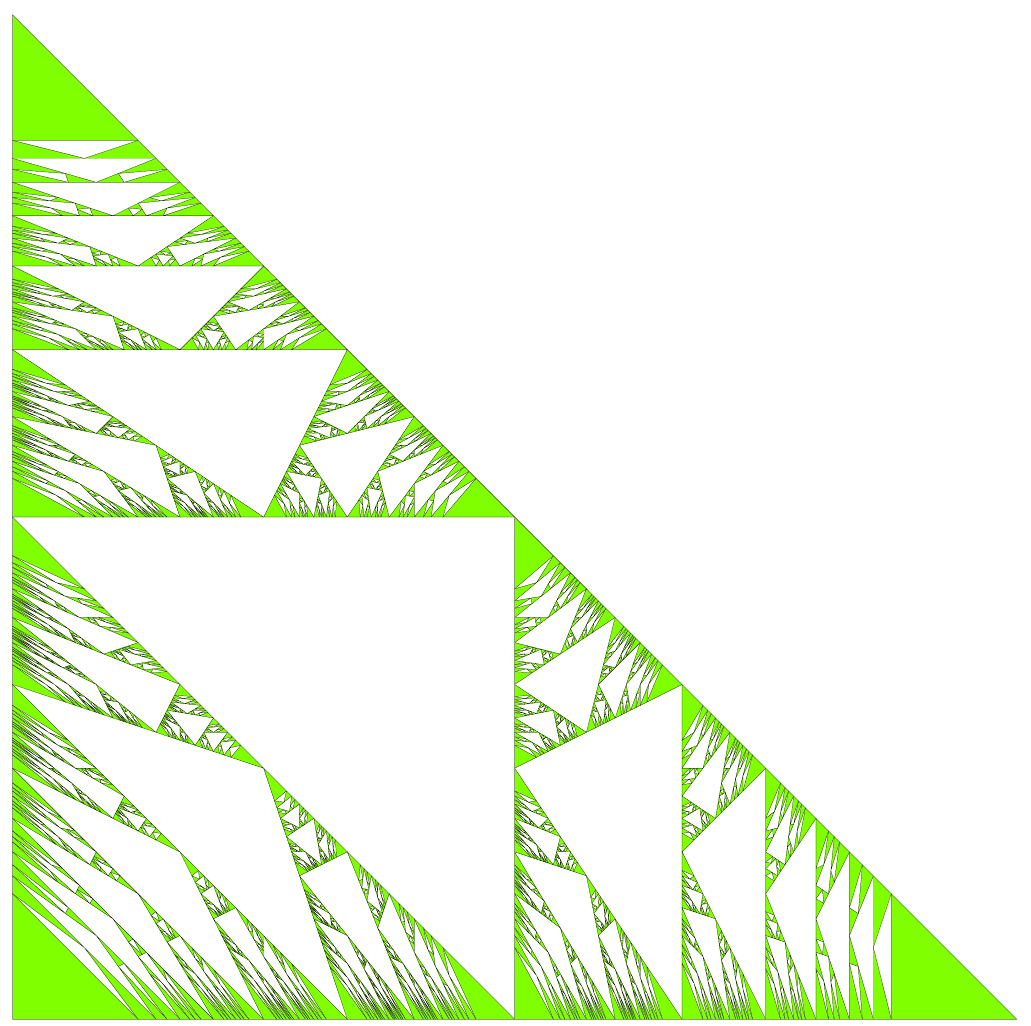}\hskip.5cm\includegraphics[width=6cm]{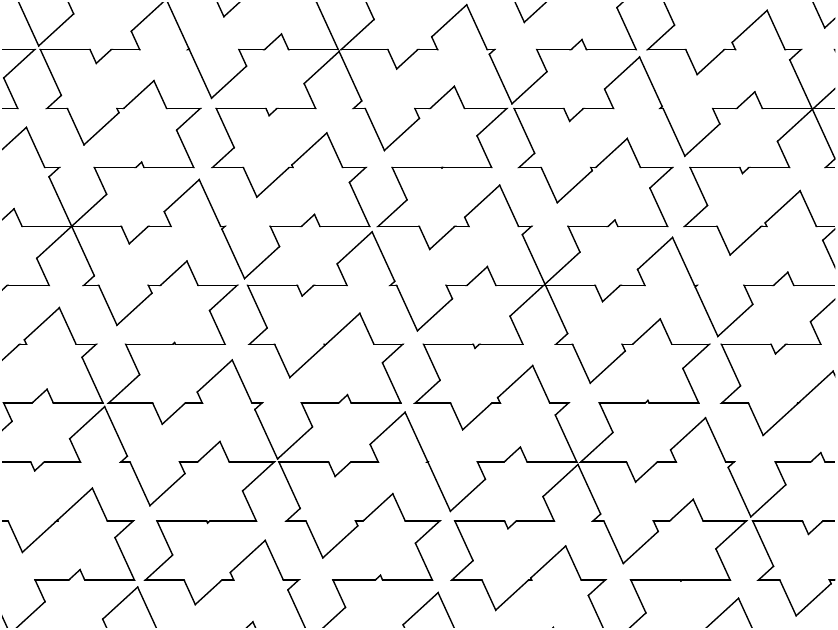}
  \caption{%
    \footnotesize                                                                                                                                                                                                               
    (left) Seven iterations of the Levitt-Yoccoz gasket.
    (right) A $B$-section for the direction $B=[\alpha^2-\alpha-1:\alpha-1:1]$, where $\alpha\simeq1.839$ is the Tribonacci constant,
    giving rise to chaotic sections on the \mucube.
  }
  \label{fig:dc}
\end{figure}
Moreover, we were able in this case to find explicit directions $B$ with $\cF_B(\mucube)$ chaotic, for example $B=[\alpha^2-\alpha-1:\alpha-1:1]$,
where $\alpha$ is the unique real solution of $\alpha^3=\alpha^2+\alpha+1$ (Tribonacci constant). Previously, the only known examples
were abstract combinatorial constructions by S.P. Tsarev in the 2-irrational case and Dynnikov in the fully irrational one (both published in~\cite{Dyn97}).

A further step in proving the Novikov conjecture for the \mucube is due to Avila, Hubert and Skripchenko in~\cite{AHS16b}.
Let us denote by $\dim_H$ the Hausdorff dimension of a set:
\begin{theorem}[Avila, Hubert and Skripchenko, 2016]
  $\dim_H\cE(\mucube)<2$.
\end{theorem}
To date no non-trivial lower bound for the Hausdorff dimension of $\cE(\mucube)$ in known but, according to a conjecture of the present
author on attractors of general self-projective semigroups (see~\cite{DL15b}), $\dim_H\cE(\mucube)\geq1.63$. A numerical rough
estimate in~\cite{DD09} gives $\dim_H\cE(\mucube)\simeq1.72$.

Finally, we mention a rigidity result on the stability zone of foliations of the torus in $\mupars$.
\begin{definition}
  A $\mupar$ is a polyhedron obtained from a $\mucube$ after rescaling over the three coordinate axes.
\end{definition}
\begin{theorem}
  Consider a family of $\mupars$ foliating the plane and let $f$ be any piecewise function having that
  family as level sets. Then $\cD(f)=\cD(\mucube)$.
\end{theorem}

\renewcommand\labelitemi{$\clubsuit$}
\noindent
{\bf Numerical objects yet unexplored:}
\begin{itemize}
\item SMs relative to surfaces of genus higher than 4.
\item SMs relative to functions whose level sets have genus higher than 4.
\item $B$-sections when $B$ is a closed 1-form with zeros. 
\end{itemize}
\subsubsection{Back to Physics}
\label{subsec:adv}
The results of Sec.~\ref{subsec:q3m} shed full light on the structure of the orbits of quasielectrons
under a magnetic field. Given a FS $M^2_g$, the set of magnetic field directions $B$ for which the
{\em chaotic} regime arise has zero measure, so it is in principle undetectable experimentally.
A generic direction gives rise to either all orbits homologous to zero or to open orbits that lie on
a rank-1 component of $M^2_g$ of topological rank 1 or 2. The first and last case are stable by small
perturbations of the direction of $B$ while the second arises only across 1-parametric families of
directions (see the schematic representation in Fig.~\ref{fig:Maltsev}).
\begin{figure}
  \centering
      \includegraphics[width=7cm]{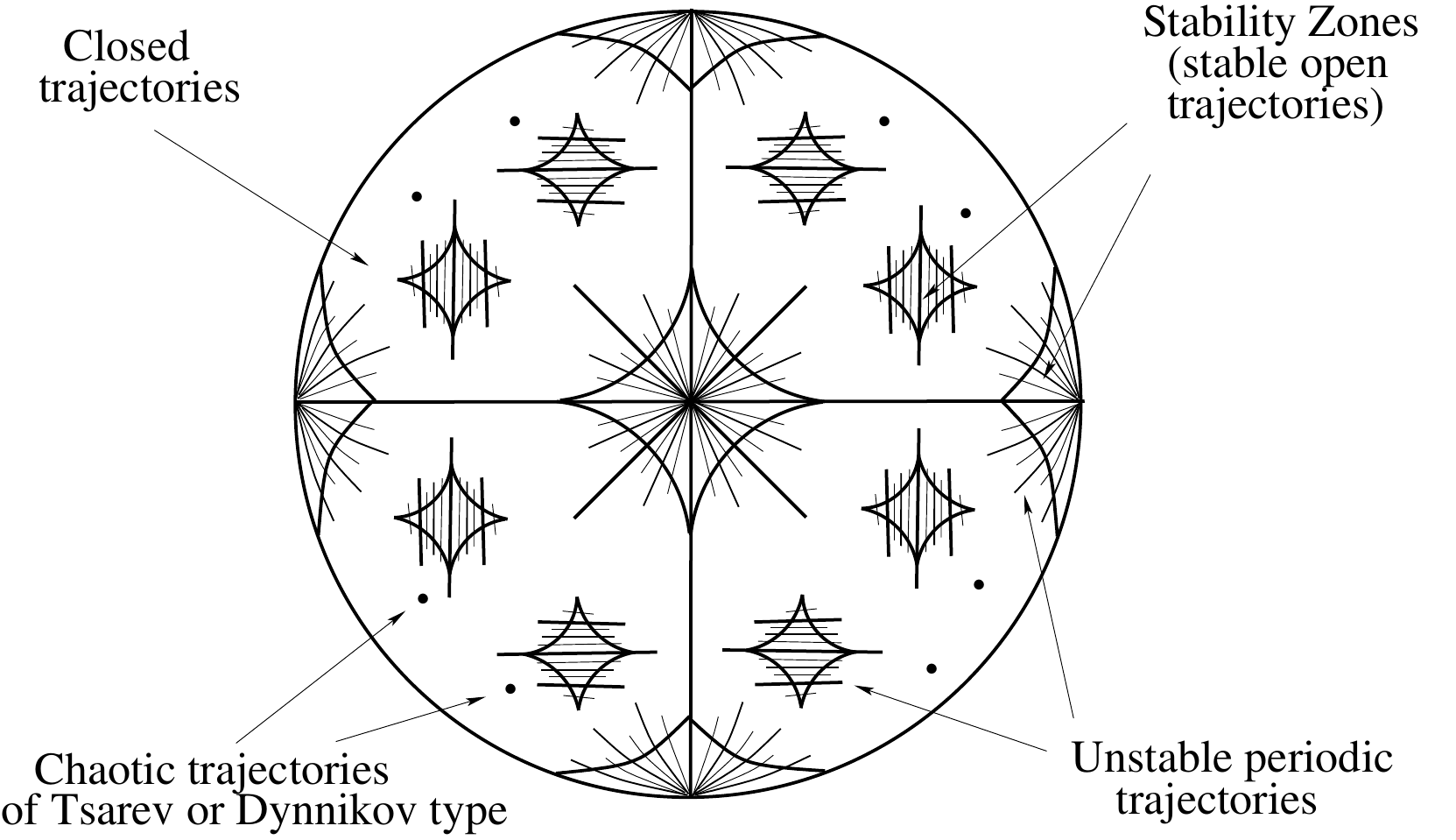} \includegraphics[width=6.3cm]{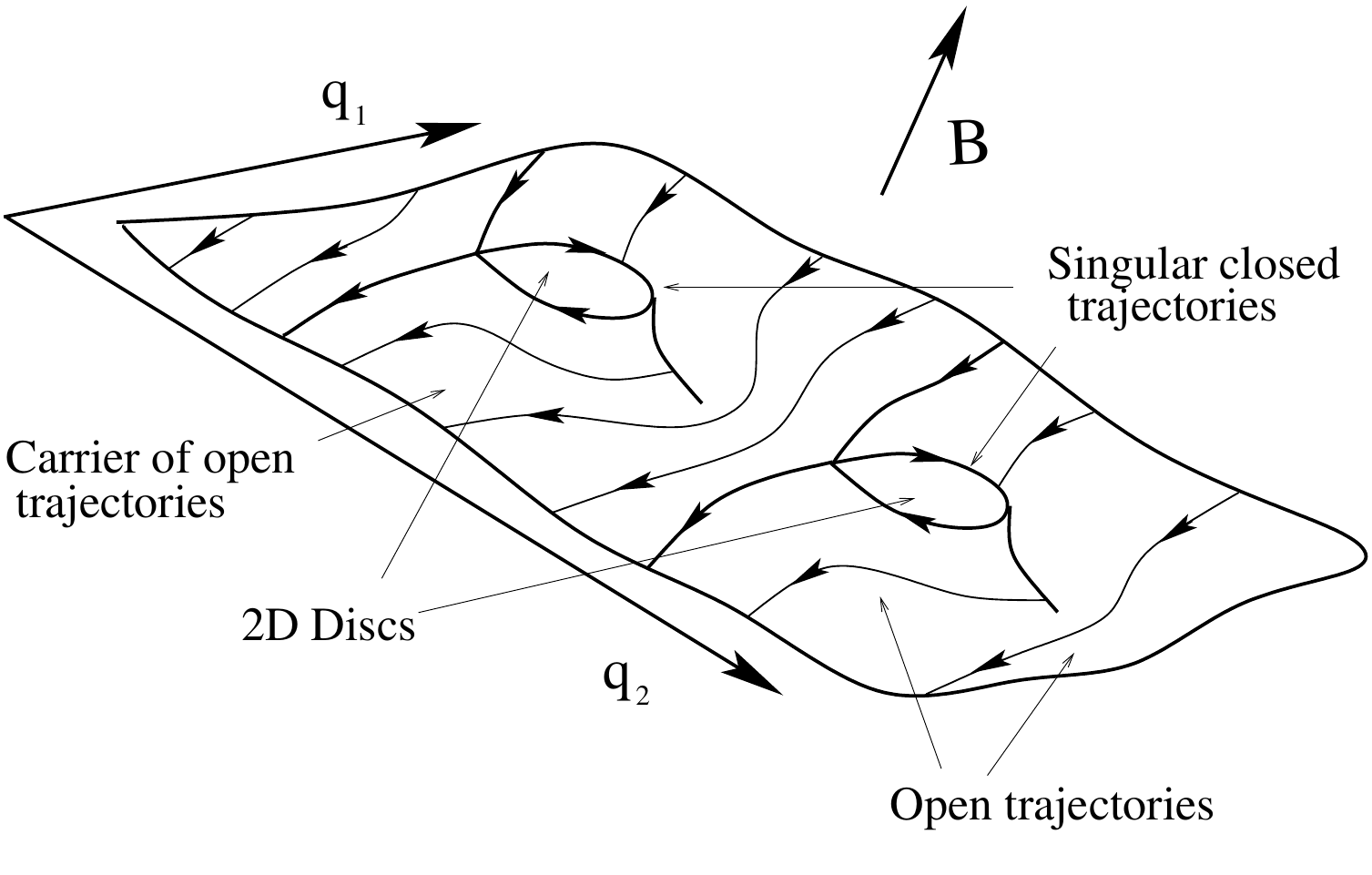}
  \caption{%
    \footnotesize                                                                                                                                                                                                               
    (left) Model of the SM of a generic surface;  (right) A genus-1 rank-2 component filled by open trajectories  (from~\cite{Mal17b}). 
  }
  \label{fig:Maltsev}
\end{figure}

This picture immediately explains qualitatively the SM obtained experimentally (see Fig.~\ref{fig:nm}): the dark islands,
where the magnetoresistance grows quadratically with the intensity of $B$, are exactly the SZ
of $\cD(M^2_g)$. The soul of each SZ dictates the asymptotics of the open orbits: if $B\in\cD_\ell(M^2_g)$,
then every nonsingular open orbit is strongly asymptotic to the direction $B\times\ell$. 
{\em The souls\hskip.05cm\footnote{From the point of view of physics, {\em souls} are Miller indices of the dual lattice.} of the SZ 
  are a hidden quantum number of purely topological origin of the system 
  and that was absolutely unknown previously to the physics community.}
Note also that now we are able to tell that some of the proposed SM (e.g. see Fig~\ref{fig:Ga60}) are qualitatively
wrong since two SZ with different souls {\em cannot overlap} for elementary topological reasons.
%
\begin{figure}
  \centering
  \includegraphics[width=4.25cm]{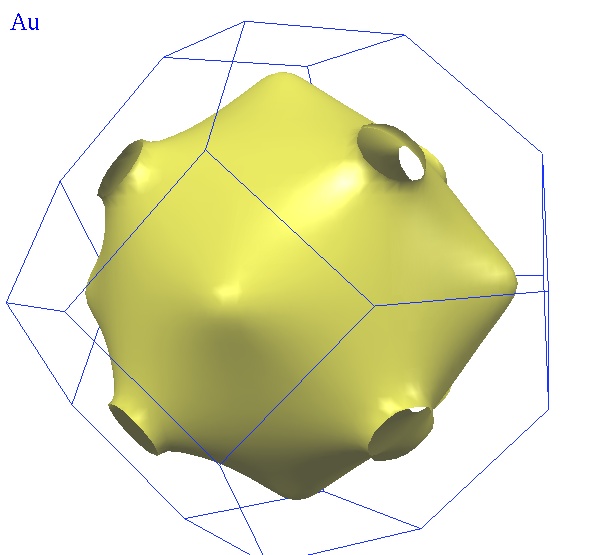}\hskip.40cm\includegraphics[width=4.25cm]{Au-exp}\hskip.40cm\includegraphics[width=4.25cm]{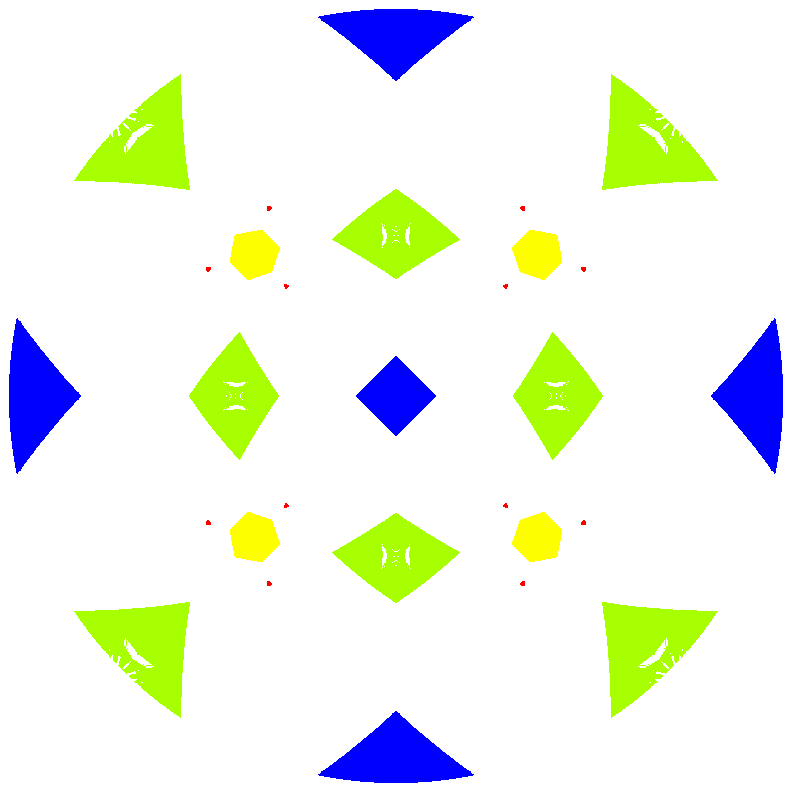}\\
  \vspace{.3cm}
  \includegraphics[width=4.25cm]{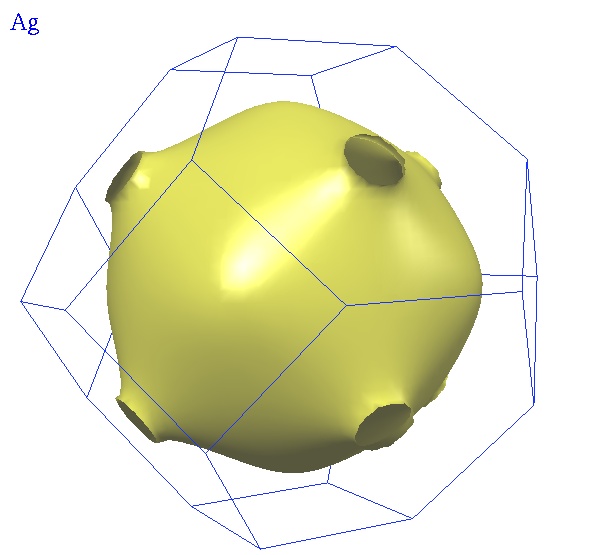}\hskip.40cm\includegraphics[width=4.25cm]{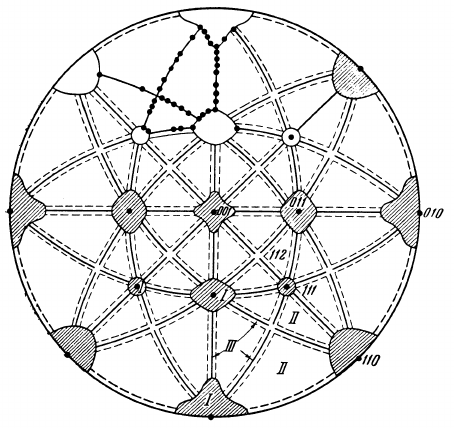}\hskip.40cm\includegraphics[width=4.25cm]{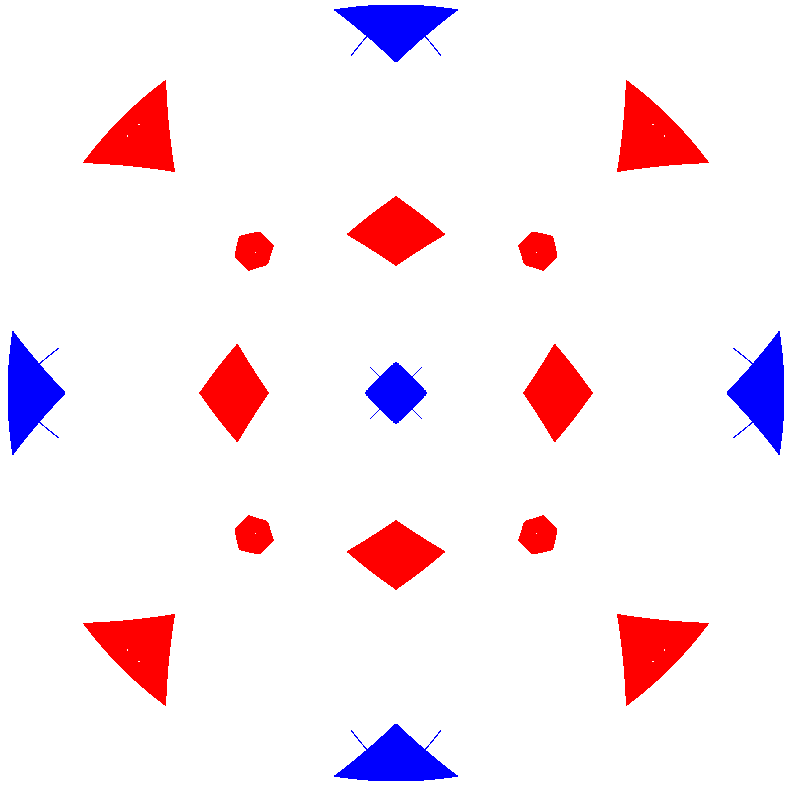}\\
  \vspace{.3cm}
  \includegraphics[width=4.25cm]{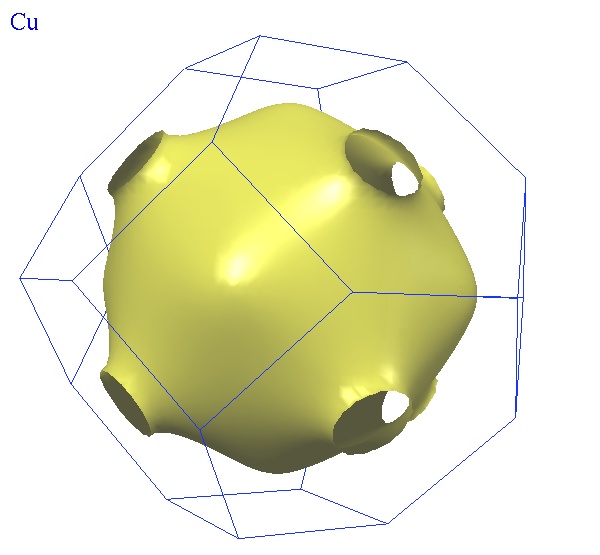}\hskip.40cm\includegraphics[width=4.25cm]{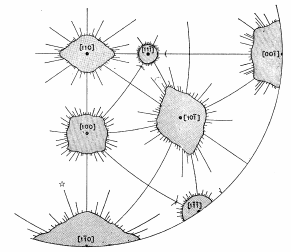}\hskip.40cm\includegraphics[width=4.25cm]{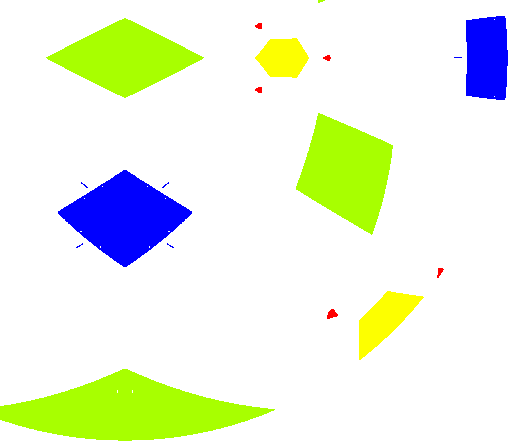}\\
  \caption{%
    \small
    (top) FS of Gold and its experimental~\cite{Ga60} and numerical~\cite{DeL04,DeL05b} Stereographic Map.
    (middle) FS of Silver and its experimental~\cite{AG62b} and numerical~\cite{DeL04,DeL05b} Stereographic Map.
    (bottom) FS of Copper and its experimental~\cite{KRBK66} and numerical Stereographic Map.
  }
\label{fig:nm}
\end{figure}

As mentioned in Sec.~\ref{subsec:mr}, the SM was found experimentally for many metals.
None of those, to the author's knowledge, has genus 3. The smallest genus arises for the FS of
noble metals, for whom $g=4$. In this case, in the integral regime the FS is generically decomposed into
2 components of genus 1 and rank 2 separated by 3 cylinders of orbits homologous to zero -- note that
the smallest genus needed to have four pairs of genus-1 components of open orbits is $g=5$.
After adding support for genus-4 surfaces to the NTC library, we extracted from~\cite{BF76} an accurate
approximation for the Fermi energy of the noble metals and finally produced~\cite{DeL04,DeL05b}
the SM shown in Fig~\ref{fig:nm} for Gold, Silver and Copper, the first metals for which a SM was produced experimentally.
As pictures show, there is a striking closeness of the experimental data with the ones coming from
the semiclassical approximation. 

The study of the physics of the magnetoresistance in a strong magnetic field has been continued in the last twenty years, at the light of the deep
mathematical advances, by A. Maltsev, jointly with Novikov and Dynnikov, in a long series of
papers~\cite{NM96,Mal97,DM97,NM98,NM03a,NM03b,NM04,NM06,Mal17a,Mal17b}.
Among the main results of his analysis we mention the observation that, under the integrable regime, the conductivity tensor has rank 1,
namely the current an flow in only one direction, and that the precises boundaries of a SZ are not detectable via standard magnetoresistance
measurements but could be determined experimentally by studying oscillatory phenomena.

\medskip
\renewcommand\labelitemi{$\clubsuit$}
\noindent
{\bf Open tasks:}
\begin{itemize}
\item Exploring SMs relative to FSs of metals for which experimental data are already in literature.
\item Encouraging the solid state community to get new more detailed data.
\end{itemize}
%
\subsection{QP functions in $\bR^2$ with 4 and more quasiperiods}
\label{subsec:q4}
In 2004, Maltsev~\cite{Mal04} showed that quasiperiodic functions with {\em any} number of quasiperiods could play a relevant role
in solid state physics. His arguments is based on a result of C.W.J. Beenakker~\cite{Bee89} that, at the end on Eighties, used a semiclassical
approximation to explain an anomalous phenomenon in the magnetoresistance of a 2-dimensional electron gas (2DEG) subject to a
weak periodic potential. A 2DEG is a semiconductor structure where the motion of electrons in one direction is somehow constrained
(and therefore quantized) so that in many phenomena only the projection of the momentum in the plane perpendicular to the constrained direction
play a role and so the system can be considered 2-dimensional. The major systems where such 2D behavior has been observed and studied
are metal-oxide-semiconductor structures, quantum well and superlattices (see~\cite{AFS82} for a thorough review on this topic).

According to Beenakker's analysis, when a constant magnetic field $B$ and an electric field $E=dV$ are applied to a 2DEG, the drift
of the center of the cyclotron orbits $r$ satisfies, in appropriate units, the equations
$$
\dot r_i = \frac{1}{\|B\|^2}\{r_i,V^{eff}_B(r)\}_B=\frac{1}{\|B\|^2}\epsilon_{ijk}\frac{\partial V^{eff}_B}{\partial r_j}B^k\,,
$$
where $V^{eff}_B$ is an effective potential depending on $V$ and $B$. Since the relevant degrees of freedom of the electron gas are in the plane
perpendicular to $B$, we can take $B=B_zdz$ and consider $V$ as a function of the two coordinates $(x,y)$ in the plane perpendicular to $B$.
The motion of electrons, therefore, is given by the level lines of $V^{eff}_B(x,y)$ and Maltsev's analysis of the problem in~\cite{Mal04,NM06} shows
that it is possible to generate quasi-periodic effective potentials with any number of quasiperiods by using several standard techniques and that,
in turn, the topological properties of the trajectories -- namely whether they are closed or open and their asymptotic directions -- can be detected
experimentally through measurements of the magnetoresistance of the 2DEG, similarly to the case of metals. Note that in this case, though,
we have more freedom since we cannot just change the direction of $B$ but also the potential, that took the role of the Fermi energy. 

We recall that the Novikov problem for $n=4$ consists in studying any of the following equivalent objects:
\begin{enumerate}
\item The level sets of a multivalued function $F=(S_1,S_2,\varepsilon):\bT^4\to\bR^3$, where $S_1$ and $S_2$
  are pseudoperiodic functions (see next section) and $\varepsilon$ is singlevalued.
\item The 2-planar sections of four-periodic hypersurfaces of $\bR^4$.
\item The leaves of the foliations induced on an embedded hypersurface $i:M^3\subset\bT^4$ by pairs of closed 1-forms $\omega_{1,2}=i^*B_{1,2}$,
  where $B_1$ and $B_2$ are non-parallel constant 1-forms on $\bT^4$.
\item The trajectories of the equations of motions corresponding to the Hamiltonian $\varepsilon$ via the Poisson bracket
  $$
  \{p_i,p_j\}_B=*\left(dp_i\wedge dp_j\wedge B_1\wedge B_2\right)
  $$
  on $\bT^4$, where $*$ is the Hodge star with respect to the Euclidean metrics and $B_1$ and $B_2$ non-parallel constant 1-forms on $\bT^4$.
\item The level sets of quasiperiodic functions with 4 quasiperiods.
\end{enumerate}
Given a fixed function $\varepsilon$, the phase space for this case is $\bR\times G_{4,2}(\bR)$. It is easy to see that, just like for $n=3$,
all rational directions in $G_{4,2}(\bR)$ give rise to all closed (possibly non-homologous to zero) sections in $\bT^4$. Indeed, a rational
$\Pi_0$ is given (in several ways) by a pair of rational directions $B_1$ and $B_2$. All leaves of $B_1$ are 3-tori $T_a$ embedded in $\bT^4$
so every level set $\varepsilon_c=\varepsilon^{-1}(c)$ restricts on $T_a$ to a triply periodic surface $M_a\subset T^a$.
The $\Pi_0$-sections of $\varepsilon^{-1}(c)$, therefore, are $B_2$-sections of $M_a$ and so, by Zorich theorem, they are all closed and lie
in a genus-1 component of $M_a$ embedded with rank 1 or 2 in $T_a$. Moreover, this result is stable with respect to small changes in
the direction of $B_2$, namely over all rational lines $r$ passing through $\Pi_0$ it is defined an $\epsilon_r$ so that, for all (rational or not)
$\Pi$ on $r$ closer to $\Pi_0$ then $\epsilon_r$,  all $\Pi$-sections have a strong asymptotic directions. In fact, the same result clearly holds
for any number $n\geq4$ of quasiperiods. The problem, though, is what happens for the directions $\Pi$ close to $\Pi_0$ outside those rational lines.

So far, only two analytical works (and no numerical one) have been published on level sets of quasiperiodic functions with more than 3 quasiperiods,
namely a first one by Novikov~\cite{Nov99} and a follow up jointly with Dynnikov~\cite{DN05}.
Their main result is the following analogue of the Zorich theorem for the case $n=3$:
\begin{definition}
  A level hypersurface $\varepsilon_c=\varepsilon^{-1}(c)$ is {\em topologically completely integrable} (TCI) for the 2-plane $\Pi\in G_{4,2}$ if all $\Pi$-sections
  are either compact or have a strong asymptotic direction. We say that $\varepsilon_c$ is {\em stable} if this property is still true after small perturbations
  of $\Pi$ and $\varepsilon$.
\end{definition}
\begin{theorem}[Novikov and Dynnikov, 2005]
  There is an open dense set of smooth functions $\cU\subset C^\infty(\bT^4)$ with the property that, for every $f\in\cU$, there is an
  open dense set $V_f\subset G_{4,2}(\bR)$ such that every $\varepsilon_c$ is a TCI for all $\Pi\in V_f$.
  Moreover, for a given $c$ and $\Pi$, there is a $R>0$ such that all open $\Pi$-sections of $\varepsilon_c$ are contained
  inside some cylinder of radius $R$. 
\end{theorem}
\begin{conjecture}[Novikov and Dynnikov, 2005]
  A generic level $\varepsilon_c$ is TCI stable for almost all $\Pi\in G_{4,2}(\bR)$.
\end{conjecture}
The case $n\geq5$ is believed to be much more complicated and no Zorich-like result is believed to hold.

\medskip\noindent 
{\bf Open tasks:} Explore numerically the cases $n=4$ and $n=5$.
%
\section{Level lines of $n-2$ quasiperiodic functions in $\bR^{n-1}$ with $n$ quasiperiods}
\label{sec:zor}
In this last section we consider yet another problem in quasiperiodic topology, this time dual in some sense to the Novikov
problem discussed in Sec.~\ref{sec:Nov}, namely the case of multivalued maps $f:\bT^n\to\bR^{n-1}$ with all but
one components singlevalued. The study of this problem was started by Zorich in 1994~\cite{Zor94,Zor97,Zor99} as a generalization
of his study of the Novikov problem for $n=3$.

In that case, Zorich proved that every open $B$-section of a surface $M^2_g\emb\bT^3$ is strongly asymptotic to a straight line for
all $B$ close enough to rational (Thm~\ref{thm:rational}). Note that, for $n=3$, the closed 1-form induced by $B$
on $M^2_g$ has irrationality degree not higher than 3, while a generic one has irrationality degree equal to the number of cycles
$2g$, namely 1-forms obtained through this construction are highly non generic. It is natural to ask whether $B$-sections
keep having an asymptotic directions when we make the setting more generic, namely by considering the embedding of
$M^2_g$ in some $\bT^n$, $n>3$. In this case, for $n$ large enough, we can get forms with full irrationality degree,
which enables one to use sophisticated dynamical systems tools that are unavailable for the $n=3$ case. 

As in the case $n=3$, we call $C_i(B)$ the components of $M^2_g$ filled by $B$-sections homologous to zero in $\bT^n$,
so that $M^2_g\setminus\cup C_i(B)$ is the union of periodic and minimal components $\cC_j(B)$
whose boundaries are all loops homotopic to zero. We denote by $g_j$ the genus of $\cC_j(B)$ and by $N$ the number of the $\cC_j(B)$.
\begin{theorem}[Zorich, 1994]
  Let $M^2_g\emb\bT^n$ be an embedding with full topological rank and assume $n\geq4g-3$.
  Then, for almost all $B\in\bR P^n$, the following properties hold:
  \begin{enumerate}
  \item $1\leq N(B)\leq g$\,.
  \item To each $\cC_j(B)$ is associated a direction $c_j\in\bR P^n$ such that
    $$
    \lim_{t\to\infty}\frac{\tilde\gamma(t)-\tilde\gamma(0)}{t}=c_j
    $$
    for every $B$-section $\gamma$ in $\cC_j(B)$, where $\tilde \gamma=\pi^{-1}_n\gamma$. 
  \item
    $$
    \limsup_{t\to\infty}\frac{\log d(\tilde\gamma(t),\ell(t))}{\log t}=\alpha(g_j)<1\,,
    $$
    where $\ell(t)$ is the straight line parallel to $c_j$ passing through $\tilde\gamma(0)$ and $1+\alpha(g_j)$ is the second Lyapunov exponent
    of the Teichmuller geodesic flow on the principal stratum of squares of holomorphic differentials on the surface
    of genus $g_j$.
  \end{enumerate}
\end{theorem}
It is noteworthy that Zorich was led by the study of this case to prove, jointly with Kontsevich~\cite{KZ03} and Avila~\cite{AV05,AV07}
a more general, important and far-reaching theorem on foliations induced by closed 1-forms on surfaces, of which the previous result is a corollary: 
\begin{theorem}[Zorich, Kontsevich and Avila, 2007]
  For almost all abelian differential $\omega$ on $M^2_g$ without maxima and minima in any connected component
  of any stratum $\cH(k_1,\cdots,k_s)$ of the moduli space of Abelian differentials with zeros of index $k_1,\dots,k_s$,
  for the foliation of the closed 1-form $\omega_0=\Re(\omega)$ there exists a complete flag of subspaces 
  $$
  V_1\subset V_2\subset \dots\subset V_g\subset H_1(M^2_g,\bR)
  $$
  with the following properties:
  \begin{enumerate}
    \item For any leaf $\gamma$ and point $p_0\in\gamma$ we have
      $$
      \lim_{t\to\infty}\frac{c_{p_0}(t)}{t}=c\,,
      $$
      where the non-zero cycle $c\in H_1(M^2_g,\bR)$ is proportional to the Poincar\'e dual of the cohomology class of $\omega_0$
      and $V_1=\sp\{c\}$.
    \item For any $\psi\in Ann(V_j)\subset H_1(M^2_g,\bR)\setminus Ann(V_{j+1})$, any leaf $\gamma$ and any point $p_0\in\gamma$,
      $$
      \limsup_{t\to\infty}\frac{\log \langle\psi,c_{p_0}(t)\rangle}{\log t}=\nu_{i+1}\;\hbox{ for all }i=1,\dots,g-1.
      $$
    \item For any $\psi\in Ann(V_g)\subset H_1(M^2_g,\bR)$, any leaf $\gamma$ and any point $p_0\in\gamma$,
      $$
      \limsup_{t\to\infty}\frac{\log \langle\psi,c_{p_0}(t)\rangle}{\log t}=0
      $$
    \item For any $\psi\in Ann(V_g)\subset H_1(M^2_g,\bR)$, $\|\psi\|=1$, any leaf $\gamma$ and any point $p_0\in\gamma$,
      $\langle\psi,c_{p_0}(t)\rangle$ is bounded from above by a constant depending only on the foliation and the norm
      chosen on the cohomology. All limits converge uniformly and the numbers $2,1+\nu_2,\dots,1+\nu_g$ are the top $g$
      exponents of the Teichmuller geodesic flow on the corresponding connected component of the stratum $\cH(k_1,\cdots,k_s)$.
  \end{enumerate}
\end{theorem}
\noindent
{\bf Open questions:} what can be said about the asymptotic direction of open leaves in the intermediate case $3<n<4g-3$?
%
%
\bibliography{refs}
\end{document}